\documentclass[a4paper,11pt, oneside]{article}
\usepackage{amsthm}
\usepackage[]{amsmath}
\usepackage{amssymb}
\usepackage{enumerate}
\usepackage{tabularx}
\usepackage[]{color}
\usepackage[left=3cm, right=3cm]{geometry}
\usepackage[colorlinks]{hyperref}
\usepackage{tikz}
\usepackage{multirow}
\usepackage{subcaption}
\usepackage{algorithm}
\usepackage{algorithmic}
\usepackage{multirow}
\usepackage{amsmath,amssymb,amsthm,mathrsfs,graphicx}
\usepackage[affil-it]{authblk}
\usepackage{authblk}

\newcommand{\bm}[1]{\boldsymbol{#1}}
\newcommand{\bmr}[1]{\bm{\mr{#1}}}
\newcommand{\lj}{[ \hspace{-2pt} [}
\newcommand{\rj}{] \hspace{-2pt} ]}
\newcommand{\mb}[1]{\mathbb{#1}}
\newcommand{\mc}[1]{\mathcal{#1}}
\newcommand{\mr}[1]{\mathrm{#1}}
\newcommand{\jump}[1]{\lj #1 \rj}
\newcommand{\aver}[1]{ \{#1\}  }

\newcommand{\tr}[1]{\ifmmode \mathrm{tr}\left( #1 \right) \else 
\text{tr} \left( #1 \right) \fi }

\renewcommand{\d}[1]{\mathrm d \boldsymbol{#1}}

\newcommand\DGnorm[1]{\| #1 \|_{{\mathrm{DG}}}}
\newcommand\DGenorm[1]{|\!|\!| #1 |\!|\!|_{{\mathrm{DG}}}}
\newcommand\DGsnorm[1]{ \| #1 \|_{*}}
\newcommand\DGesnorm[1]{ |\!|\!| #1 |\!|\!|_{*}}

\newcommand\comment[1]{}

\def\MTh{\mc{T}_h}
\def\MThi{\mc{T}_h^i}
\def\MThic{\mc{T}_h^{i, \circ}}
\def\MTho{\mc{T}_h^0}
\def\MThoc{\mc{T}_h^{0, \circ}}
\def\MThl{\mc{T}_h^1}
\def\MThlc{\mc{T}_h^{1, \circ}}

\def\MEh{\mc{E}_h}
\def\MEhb{\mc{E}_h^b}
\def\MEhc{\mc{E}_h^\circ}
\def\MEho{\mc{E}_h^0}
\def\MEhoc{\mc{E}_h^{0, \circ}}
\def\MEhl{\mc{E}_h^1}
\def\MEhlc{\mc{E}_h^{1, \circ}}

\def\MThG{\mc{T}_h^{\Gamma}}
\def\MEhG{\mc{E}_h^{\Gamma}}
\def\un{\bm{\mr{n}}}

\newtheorem{assumption}{Assumption}
\newtheorem{theorem}{Theorem}
\newtheorem{lemma}{Lemma}

\newtheorem{remark}{Remark}

\begin{document}

\title
{An unfitted finite element method by direct
extension for elliptic problems on domains with curved boundaries
and interfaces}

\author
{Fanyi Yang \thanks{ Email: yangfanyi@scu.edu.cn } }
\author
{Xiaoping Xie \thanks{Corresponding author. Email: xpxie@scu.edu.cn} }

\affil{School of Mathematics, Sichuan University, Chengdu 610064, China}
\date{}

\maketitle

\begin{abstract}
  We propose and analyze an unfitted finite element method for solving
  elliptic problems on domains with curved boundaries and
  interfaces.  The approximation space on the whole domain is obtained
  by the direct extension of the finite element space defined on
  interior elements, in the sense that there is no degree of freedom
  locating in boundary/interface elements. The boundary/jump
  conditions  are imposed in a weak sense in the scheme.  The method
  is shown to be stable without any mesh adjustment or any special
  stabilization. Optimal convergence rates under the $L^2$ norm and
  the energy norm are derived.  Numerical results in both two and
  three dimensions are presented to illustrate the accuracy and the
  robustness of the method.

\noindent \textbf{keywords}: elliptic   problems; curved boundary;
 interface problems; finite element method; Nitsche's method;
unfitted mesh
\end{abstract}

\section{Introduction}
\label{sec_introduction}
In recent two decades,  unfitted finite element methods have become
widely used tools in the numerical analysis of problems with
interfaces  and  complex geometries 
\cite{babuska2011Stable,belytschko2009review,Bordas2017geometrically,
Gurkan2019stabilized, Hansbo2002unfittedFEM,Lehrenfeld2018analysis,
lizhilin1998immersed,lin2015partially,
strouboulis2000design,Thomas2010The}. 
For such kinds of problems, the generation of the body-fitted meshes
is usually a very {challenging} and
time-consuming task, especially in three dimensions. 
The unfitted  methods avoid the task to generate  high quality meshes
for representing the domain geometries accurately, due to the use of
meshes independent of the interfaces and  domain boundaries and   the
use of certain enrichment of finite element basis functions
characterizing the solution singularities or discontinuities.

%

%

In \cite{Hansbo2002unfittedFEM}, Hansbo and Hansbo proposed an
unfitted finite element method for  elliptic interface problems. The
numerical solution comes from two separate linear finite element
spaces and the jump conditions are weakly enforced by Nitsche's
method. This idea has been a popular discretization for interface
problems and has also been applied to many other interface problems,
see \cite{Burman2012ficticious,Burman2014fictitious, Hansbo2014cut}
and the references therein for further advances. This method can also
be written into the framework of  extended finite element method by a
Heaviside enrichment
\cite{Areias2006letter,belytschko2009review,Thomas2010The}.  We note
that for penalty methods, the small cuts of the mesh have to be
treated carefully, which may adversely effect the conditioning of the
method and even hamper the convergence \cite{Burman2021unfitted,
Prenter2018note}.  In \cite{Johansson2013high}, Johansson and Larson
proposed an unfitted high-order discontinuous Galerkin method on
structured grids, where they constructed large extended elements to
cure the issue of the small cuts and obtain the stability near the
interface. Similar ideas of merging elements for interface problems
can also be found in \cite{Burman2021unfitted,Huang2017unfitted,
Li2018interface}.  Another popular unfitted method is the cut finite
element method \cite{Burman2015cutfem}, which is a variation of the
extended finite element method. This method involves the ghost penalty
technique \cite{Burman2010ghost} to guarantee the stability of the
scheme. In addition, Massing and G\"urkan develop a framework
combining the cut finite element method and the discontinuous Galerkin
method \cite{Gurkan2019stabilized}. We refer to
\cite{Bordas2017geometrically, Burman2015cutfem, Burman2020stable,
Gurkan2020stabilized, 
HanChenWangXie2019Extended,Massing2014stabilized} and the references
therein for some recent applications of the cut finite element method.
In \cite{Lehrenfeld2016high}, Lehrenfeld introduced a high order
unfitted finite element method based on isoparametric mappings, where
the piecewise interface is mapped approximately onto the zero level
set of a high-order approximation of the level set function. We refer
to \cite{
Lehrenfeld2018analysis} for an analysis
of more details to this method.

In this article, we propose a new unfitted finite element method for
second order elliptic problems on domains with curved boundaries
and  interfaces. The novelty of this method lies in that the
approximation space is obtained by the direct extension of a common
finite element space.  We first define a standard finite element space
on the set of all interior elements which are not cut by the domain
boundary/interface. Then an extension operator is introduced for this
space. This operator defines the polynomials on cut elements by
directly extending the polynomials defined on some interior
neighbouring elements. Then the approximation
space is obtained from the extension operator. 
In the discrete schemes, a
symmetric interior penalty method is adopted, and the boundary/jump
conditions on the interface are weakly satisfied by Nitsche's method.
We derive optimal error estimates   under the energy norm and
the $L^2$ norm, and we give upper bounds of the condition numbers of the final linear
systems. The curved boundary and the interface are allowed to intersect
the mesh arbitrarily in our method.  We note that the idea of
associating elements that have small intersections with neighbouring
interior elements can also be found in, for example,
\cite{Johansson2013high, Gurkan2019stabilized, Huang2017unfitted}. But
different from the previous methods,   the proposed method has no
degrees of freedom locating in cut elements, and does not need  any
mesh adjustment or  extra stabilization mechanism. The implementation
of our method is very simple and the method can easily achieve
high-order accuracy.  We conduct a series of numerical experiments in
two and three dimensions to illustrate the convergence behaviour.

The rest of this article is organized as follows. In Section
\ref{sec_preliminaries}, we introduce   notations   and   prove some
basic properties for the approximation space.  We show the unfitted
finite element method for   the elliptic  problem on a curved domain
and  the elliptic interface problem in Sections \ref{sec_elliptic} and
\ref{sec_interface}, respectively,  {and we} derive  
optimal error estimates, and  give    upper bounds of the condition
numbers of the discrete systems.  
In Section \ref{sec_numericalresults}, we perform some numerical tests
to   confirm the optimal convergence rates and show the robustness of
the proposed method. Finally, we make a conclusion in Section
\ref{sec_conclusion}.

\section{Preliminaries} 
\label{sec_preliminaries}
Let $\Omega \subset \mb{R}^d(d = 2, 3)$ be a convex polygonal
(polyhedral) domain with boundary $\partial \Omega$. Let $\Omega_0
\Subset \Omega$ be an open subdomain with $C^2$-smooth or convex
polygonal (polyhedral) boundary. We denote by $\Gamma := \partial
\Omega_0 $ the topological boundary. Let $\MTh$ be a background mesh
which is 
a quasi-uniform 
triangulation of the domain
$\Omega$ into simplexes (see Fig.~\ref{fig_MThMTh0} for the example
that $\Gamma$ is a circle). We denote by $\MEh$ the collection of all
$d - 1$ dimensional faces in $\MTh$. We further decompose $\MEh$ into
$\MEh = \MEhb \cup \MEhc$, where $\MEhb$ and $\MEhc$ consist of
boundary faces and interior faces, respectively.  For any element $K
\in \MTh$ and any face $e \in \MTh$, we denote by $h_K$ and $h_e$ 
their diameters, respectively. The mesh size $h$ is defined as $h :=
\max_{K \in \MTh} h_K$. The quasi-uniformity of $\MTh$ is in the sense
of that there exists a constant $C$ such that $h \leq C\rho_K$ for any
element $K$, here $\rho_K$ is the radius of the largest ball inscribed
in $K$.

Since $\Omega_0$ can be a curved domain, we set (see Fig.~\ref{fig_MThMTh0})
\begin{displaymath}
  \MTho := \{ K \in \MTh \ | \ K \cap \Omega_0 \neq \varnothing\},
  \quad \MThoc :=  \{ K \in \MTho \ | \ K \subset \Omega_0 \}.
\end{displaymath}
Clearly, $\MTho$ is the minimal subset of $\MTh$ that just covers the
domain $\overline{\Omega}_0$, and $\MThoc$ is the set of elements
which are inside the domain $\Omega_0$.
\begin{figure}[htp]
  \centering
  \begin{minipage}[t]{0.46\textwidth}
    \centering
    \begin{tikzpicture}[scale=2.2]
      \input{figure/MTh.tex}
      \draw[blue, thick] (0, 0) circle [radius=0.6]; 
    \end{tikzpicture}
  \end{minipage} 
  \begin{minipage}[t]{0.46\textwidth}
    \centering
    \begin{tikzpicture}[scale=2.2]
      \input{figure/MTh0.tex}
      \draw[blue, thick] (0, 0) circle [radius=0.6]; 
    \end{tikzpicture}
  \end{minipage}
  \caption{The background mesh $\MTh$ (left) / the mesh $\MTh^0$ that
  covers the domain $\Omega_0$ (right) and  the elements in $\MThG$
  (orange).}
  \label{fig_MThMTh0}
\end{figure}
For the set $\MEh$, we let
\begin{displaymath}
  \MEho := \{ e \in \MEh \ | \ e \cap \Omega_0 \neq \varnothing\}, 
  \quad \MEhoc :=  \{ e \in \MEho \ | \ e \subset \Omega_0 \}
\end{displaymath}
be the sets of $d-1$ dimensional faces corresponding to $\MTho$ and
$\MThoc$, respectively.  Moreover, we denote by $\MThG$ and $\MEhG$
the sets of the elements and faces that are cut by $\Gamma$ (see
Fig.~\ref{fig_MThMTh0}), respectively: 
\begin{displaymath}
  \MThG := \{K \in \MTh \ | \ K \cap \Gamma \neq \varnothing \}, 
  \quad \MEhG := \{ e \in \MEh \ | \ e \cap \Gamma \neq \varnothing
  \}.
\end{displaymath}
Obviously, we have $\MThG = \MTho \backslash \MThoc$ and $\MEhG =
\MEho \backslash \MEhoc$. For any element $K \in \MThG$, we define
the curve $\Gamma_K := K \cap \Gamma$. 

We make following natural geometrical assumptions on the background
mesh: 
\begin{assumption}
  For any cut face $e \in \MEhG$, the intersection $e \cap \Gamma$ is
  simply connected; that is, $\Gamma$ does not cross a face multiple
  times. 
  \label{as_mesh1}
\end{assumption}
\begin{assumption}
  For any element $K \in \MThG$, there is an element $K^{\circ}
  \in \Delta(K) \cap \MThoc$, where $\Delta(K)
  := \{ K' \in \MTh \ | \ \overline{K'} \cap \overline{K} \neq
  \varnothing\}$ denotes the set of elements that touch $K$.
  \label{as_mesh2}
\end{assumption}
\begin{remark}
The above assumptions are widely used in   interface problems
\cite{Hansbo2002discontinuous, Massing2014stabilized, Wu2012unfitted},
which ensure the curved boundary $\Gamma$ is well-resolved by the
mesh.  We note that if the mesh is fine enough,  Assumptions
\ref{as_mesh1} and \ref{as_mesh2} can always be fulfilled. 
\end{remark}

From the quasi-uniformity of the mesh, there
exists a constant $C_{\Delta}$ independent of $h$ such that for any
element $K \in \MTh$, there is a ball $B(\bm{x}_K, C_\Delta h_K)$
satisfying $\Delta(K) \subset B(\bm{x}_K, C_\Delta h_K)$, where
$\bm{x}_K$ is the barycenter of $K$ and $B(\bm{z}, r)$ denotes the
ball centered at $\bm{z}$ with radius $r$. Moreover, 
let  $\Omega^*$ be an open bounded domain, independent of the mesh size $h$ and
$\Gamma$, which includes the union of
all balls $B(\bm{x}_K, C_\Delta h_K)$ $(\forall K \in \MTh)$, that is,
$B(\bm{x}_K, C_\Delta h_K) \in \Omega^*$ for any $K \in \MTh$.

Next, we introduce the     jump and   average operators   which are
widely used in the discontinuous Galerkin framework.  Let $e \in
\MEhc$ be any interior face   shared by two neighbouring elements
$K^+$ and $K^-$, with the unit outward normal vectors $\un^+$ and
$\un^-$ along $e$, respectively. For any piecewise smooth
scalar-valued function $v$ and  piecewise smooth vector-valued
function $\bm{q}$, the jump operator $\jump{\cdot}$ is defined as 
\begin{displaymath}
  \jump{v}|_e := (v|_{K^+})|_e \un^+ + (v|_{K^-})|_e \un^-, \quad
  \jump{\bm{q}}|_e := (\bm{q}|_{K^+})|_e \cdot \un^+ +( \bm{q}|_{K^-})|_e \cdot
  \un^-, 
\end{displaymath}
and the average operator $\aver{\cdot}$ is defined as 
\begin{displaymath}
  \aver{v}|_e := \frac{1}{2}\left((v|_{K^+} )|_e+ (v|_{K^-})|_e\right), \quad
  \aver{\bm{q}}|_e := \frac{1}{2}\left((\bm{q}|_{K^+})|_e +( \bm{q}|_{K^-})|_e\right). 
\end{displaymath}
On a boundary face $e \in \MEhb$  
with the unit outward normal vector $\un$, we define 
\begin{displaymath}
  \aver{v}|_e := v|_e, \quad \jump{v}|_e := v|_e \un, \quad
  \aver{\bm{q}}|_e := \bm{q}|_e, \quad \jump{\bm{q}}|_e := \bm{q}|_e \cdot
  \un.
\end{displaymath}
We will also employ the jump operator $\jump{\cdot}$ and the average
$\aver{\cdot}$ on $\Gamma$, that is,
\begin{equation}
  \jump{v}|_{\Gamma}, \quad \aver{v}|_{\Gamma},
  \quad \jump{\bm{q}}|_{\Gamma}, \quad \aver{\bm{q}}|_{\Gamma},
  \label{eq_traceG}
\end{equation}
and their definitions will be given later for specific problems.

Throughout this paper,  we denote by $C$ and $C$ with subscripts
the generic positive constants that may vary between lines but are
independent of the mesh size $h$ and how $\Gamma$ cuts the
mesh $\MTh$. For a bounded domain $D$, we follow the standard
notations of the Sobolev spaces $L^2(D)$, $H^r(D)(r \geq 0)$ and their
corresponding inner products, norms and semi-norms. For the partition
$\MTh$, the notations of broken Sobolev spaces $L^2(\MTh)$,
$H^1(\MTh)$ are also used as well as their associated inner products
and broken Sobolev norms.

We follow three steps to  give the definition of the approximation
space $V_{h, 0}^m$ with respect to the partition $\MTho$. 

Step 1. Let $V_{h, 0}^{m, \circ}$ be the space of piecewise polynomials
  of degree $m \geq 1$ on   $\MThoc$. Here $V_{h,
0}^{m, \circ}$ can be the standard $C^0$ finite element space or the
discontinuous finite element space, i.e.
\begin{displaymath}
  V_{h, 0}^{m, \circ} = \{ v_h \in C(\MThoc) \ | \ v_h|_K \in
  \mb{P}_m(K), \ \forall K \in \MThoc \}, 
\end{displaymath}
or
\begin{displaymath}
  V_{h, 0}^{m, \circ} = \{ v_h \in L^2(\MThoc) \ | \ v_h|_K \in
  \mb{P}_m(K), \ \forall K \in \MThoc \},
\end{displaymath}
where  $\mb{P}_m(K)$ denotes the set of polynomials of degree $m$
defined on $K$.
  
Step 2.  We extend the space $V_{h, 0}^{m, \circ}$ to the mesh $\MTho$
by introducing an extension operator $E_{h, 0}$. To this end, for
every element $K \in \MTh$, we define a local extension operator
\begin{equation}
  \begin{aligned}
    E_K: \mb{P}_m(K) &\rightarrow \mb{P}_m(B(\bm{x}_K, C_{\Delta} h_K)
    ), \\
    v & \mapsto E_K v, \\
  \end{aligned} \quad v|_K = (E_K v)|_K.
  \label{eq_Ek}
\end{equation}
For any $v \in \mb{P}_m(K)$, $E_Kv$ is a polynomial defined on the
ball  $B(\bm{x}_K, C_{\Delta} h_K)$ and has the same expression as
$v$. Then the operator $E_{h, 0}$ is defined in a piecewise manner:
for any   $K \in \MTho$ and $v_h \in V_{h, 0}^{m, \circ}$,
\begin{equation}
  (E_{h, 0} v_h)|_K := \begin{cases}
    v_h|_K, & \forall K \in \MThoc, \\
    (E_{K^{\circ}} v)|_{K}, & \forall K \in \MThG, \\
  \end{cases} 
  \label{eq_Eh}
\end{equation}
 where $K^\circ$ is defined in
Assumption \ref{as_mesh2}. Note that for any cut element $K \in \MThG$, the
operator $E_{h, 0}$ extends  polynomials of degree $m$ from the assigned interior
element $K^\circ$ to   $K$.

Step 3.  We define the approximation space $V_{h, 0}^m$   as the image space of
the operator $E_{h, 0}$,
\begin{displaymath}
V_{h, 0}^m := \{
E_{h, 0} v_h\ | \ \forall v_h \in V_{h, 0}^{m, \circ} \}.
\end{displaymath}

From \eqref{eq_Eh}, it can be seen that $V_{h, 0}^m$ is a
piecewise polynomial space and shares the same degrees of freedom of
the space $V_{h, 0}^{m, \circ}$, which implies that all degrees of freedom
of $V_{h, 0}^m$ locate inside the domain $\Omega_0$. 

Let   $I_{h, 0}$ be
the corresponding Lagrange interpolation operator of the space $V_{h,
0}^{m, \circ}$ and recall that $\Omega^*$ is an open bounded domain  including the union of
all balls $B(\bm{x}_K, C_\Delta h_K)$ $(\forall  K \in \MTh)$.  Then the following lemma shows   the
  approximation property of the space $V_{h, 0}^m$.
\begin{lemma}
  For any element $K \in \MThoc$, there exists a constant $C$ such
  that 
  \begin{equation}
    \|u - I_{h, 0} u\|_{H^q(K)} \leq C h_K^{m + 1 - q} \| u
    \|_{H^{m+1}(K)}, \quad q = 0, 1, \quad  \forall u \in
    H^{m+1}(\Omega^*),
    \label{eq_Vhoaperror1}
  \end{equation}
  and for any element $K \in \MThG$, there exists a constant $C$ such
  that 
  \begin{equation}
    \|u - E_{h, 0}(I_{h, 0} u)\|_{H^q(K)} \leq C h_K^{m + 1 - q} \| u
    \|_{H^{m+1}(B(\bm{x}_{K^\circ}, C_\Delta h_{K^{\circ}}))}, \quad q
    = 0, 1, \quad  \forall u \in H^{m+1}(\Omega^*).
    \label{eq_Vhoaperror2}
  \end{equation}
  \label{le_Vhoaperror}
\end{lemma}
\begin{proof}
  It is sufficient to verify the estimate \eqref{eq_Vhoaperror2}, since
  the estimate \eqref{eq_Vhoaperror1} is standard. For the ball
  $B(\bm{x}_{K^\circ}, C_\Delta h_{K^{\circ}})$, there exists a
  polynomial $v_h \in \mb{P}_m(B(\bm{x}_{K^\circ}, C_\Delta
  h_{K^{\circ}}))$ such that \cite{Brenner2007mathematical} 
  \begin{displaymath}
    \|u - v_h \|_{H^q(B(\bm{x}_{K^\circ}, C_\Delta h_{K^{\circ}}))}
    \leq C h_{K^\circ}^{m + 1 - q} \| u
    \|_{H^{m+1}(B(\bm{x}_{K^\circ}, C_\Delta h_{K^{\circ}}))}. 
  \end{displaymath}
  Thus, we have that 
  \begin{displaymath}
    \begin{aligned}
      \|u - E_{h, 0}(I_{h, 0} u)\|_{H^q(K)} \leq \|u - v_h \|_{H^q(K)}  +
      \|v_h - E_{h, 0}(I_{h, 0} u) \|_{H^q(K)}. 
    \end{aligned}
  \end{displaymath}
  From the mesh regularity, there exists a constant $C_1$ such that
  $$C_\Delta \leq (C_\Delta h_{K^\circ}) / \rho_{K^\circ} \leq C_1.$$
  {Considering the norm equivalence between $\|
  \cdot \|_{L^2(B(\bm{x}_{K^\circ}, C_1))}$ and $ \| \cdot
  \|_{L^2(B(\bm{x}_{K^\circ}, 1))}$ for the space $\mb{P}_m(\cdot)$ and
  the affine mapping from $B(\bm{x}_{K^\circ}, 1)$
  to the $B(\bm{x}_{K^\circ}, \rho_{K^\circ})$,} there holds 
  \begin{displaymath}
    \|q_h \|_{H^q(B(\bm{x}_{K^\circ}, C_\Delta h_{K^\circ}))}
    \leq C \|q_h \|_{H^q(B(\bm{x}_{K^\circ},
    \rho_{K^\circ}))}, \quad \forall q_h \in \mb{P}_m(
    B(\bm{x}_{K^\circ}, C_\Delta h_{K^\circ})).
  \end{displaymath}
  Combining $h_{K^\circ} \leq C h_K$, the above result brings us that
  \begin{displaymath}
    \begin{aligned}
      \|v_h - E_{h, 0}(I_{h, 0} u) \|_{H^q(K)} & = \|E_{h, 0}( v_h -
      I_{h, 0} u) \|_{H^q(K)} \leq \|E_{h, 0}( v_h - I_{h, 0} u)
      \|_{H^q( B(\bm{x}_{K^\circ}, C h_{K^{\circ}}) )} \\
      & \leq C \|v_h - I_{h, 0} u \|_{H^q( B(\bm{x}_{K^\circ},
      \rho_{K^{\circ}}) )} \leq C \|v_h - I_{h, 0} u \|_{H^q( K^\circ
      )} \\
      & \leq C \left( \|u - v_h \|_{H^q(K^\circ)} + \|u - I_{h, 0} u
      \|_{ H^q(K^\circ)} \right) \\
      & \leq C h_K^{m+1 - q} \| u \|_{H^{m+1}(B(\bm{x}_{K^\circ}, C
      h_{K^{\circ}}))},
    \end{aligned}
  \end{displaymath}
  which completes the proof. 
\end{proof}
We have given the definition and the basic property of the
approximation space. The implementation of the space is the same as
the common finite element spaces, which is very simple and does not
need any strategy for adjusting the mesh to eliminate the effects of
the small cuts. Thus the curve $\Gamma$ is allowed to intersect the
partition in an arbitrary fashion. In next two sections, we will apply
the space $V_{h, 0}^m$ to solve the elliptic  problem on a curved
domain and the elliptic   interface problem, respectively.

\section{Approximation to Elliptic   Problem on  Curved Domain}
\label{sec_elliptic}
In this section, we are concerned with the model boundary problem
defined on the curved domain $\Omega_0$: seek $u$ such that 
\begin{equation}
  \begin{aligned}
    -\Delta u & = f, && \text{in } \Omega_0, \\
    u & = g, && \text{on } \Gamma. \\
  \end{aligned}
  \label{eq_elliptic}
\end{equation}
We assume $f \in L^2(\Omega_0)$ and $g \in H^{3/2}(\Gamma)$. Then the
problem \eqref{eq_elliptic} admits a unique solution $u \in
H^2(\Omega_0)$ from the standard regularity result
\cite{Girault1986finite}. 

For this problem, the mesh $\MTh$ can be
regarded as a background mesh that entirely covers the domain
$\overline{\Omega}_0$, and $\MTho$ is the minimal subsets of $\MTh$
covering $\Omega_0$. The trace operators in \eqref{eq_traceG} for this
problem are specified as 
\begin{displaymath} \aver{v}|_{\Gamma_K} := v|_{\Gamma_K}, \quad
  \jump{v}|_{\Gamma_K} := v|_{\Gamma_K} \un, \quad \aver{\bm{q}}|_{\Gamma_K} :=
  \bm{q}|_{\Gamma_K}, \quad \jump{\bm{q}}|_{\Gamma_K} := \bm{q}|_{\Gamma_K} \cdot \un, 
\end{displaymath}
for any $K \in \MThG$, where $\un$ denotes the unit outward normal
vector on $\Gamma$. 

We solve the problem \eqref{eq_elliptic} by the space
$V_{h, 0}^m$, and the numerical solution is sought by the following
discrete variational form: find $u_h \in V_{h, 0}^m$
such that
\begin{equation}
  a_h(u_h, v_h) = l_h(v_h), \quad \forall v_h \in V_{h, 0}^m,
  \label{eq_dvariation}
\end{equation}
where the bilinear form $a_h(\cdot, \cdot)$ takes the form
\begin{eqnarray} \label{eq_bilinear}
    a_h(u_h, v_h) &:= & \sum_{K \in \MTho} \int_{K \cap \Omega_0} 
    \nabla u_h \cdot \nabla v_h \d{x}  \\
    &&- \sum_{e \in \MEho} \int_{e \cap \Omega_0} \left(\aver{\nabla
    u_h} \cdot \jump{v_h}   
    {+}  \aver{\nabla v_h} \cdot \jump{u_h}  
    {-}
    \frac{\mu}{h_e} \jump{u_h} \cdot \jump{v_h} \right)d{s}
    \nonumber\\
     &&- \sum_{K \in \MThG} \int_{\Gamma_K} \left(\aver{\nabla u_h} \cdot
     \jump{v_h} 
     {+} \aver{\nabla
     v_h} \cdot \jump{u_h} 
     {-}
    \frac{\mu}{h_K} \jump{u_h} \cdot \jump{v_h} \right)d{s} , \nonumber
 \end{eqnarray}
where $\mu$ is the positive penalty parameter and will be specified
later on. The linear form $l_h(\cdot)$ reads
\begin{equation}
  \begin{aligned}
    l_h(v_h) := \sum_{K \in \MTho} \int_{K \cap \Omega_0} fv_h d{x}  -
    \sum_{K \in \MThG} \int_{\Gamma_K} \aver{\nabla v_h} \cdot \un
   {g} \d{s} + \sum_{K \in \MThG} \int_{\Gamma_K}
    \frac{\mu}{h_K} \jump{v_h} \cdot \un g\d{s}.
  \end{aligned}
  \label{eq_linearform}
\end{equation}
Notice that   \eqref{eq_bilinear} is suitable for both cases that 
$V_{h, 0}^{m, \circ}$ is the discontinuous piecewise polynomial space
or the $C^0$ finite element space.
If $V_{h, 0}^{m, \circ}$ is chosen to be the continuous space,
$a_h(\cdot, \cdot)$ can be further simplified, see Remark
\ref{re_c0space}.
\begin{remark}
 For the $C^0$ finite element space $V_{h, 0}^{m, \circ}$, the terms
 defined on $\MEho$ 
 \begin{displaymath}
   \sum_{e \in \MEho} \int_{e \cap \Omega_0} \aver{\nabla u_h} \cdot
   \jump{v_h} \d{s}, \quad \sum_{e \in \MEho} \int_{e \cap \Omega_0}
   \aver{\nabla v_h} \cdot \jump{u_h} \d{s}, \quad \sum_{e \in \MEho}
   \int_{e \cap \Omega_0} \mu h_e^{-1} \jump{u_h} \cdot \jump{v_h}
   \d{s},
 \end{displaymath}
 are reduced to 
 \begin{displaymath}
   \sum_{e \in \MEhG} \int_{e \cap \Omega_0} \aver{\nabla u_h} \cdot
   \jump{v_h} \d{s}, \quad \sum_{e \in \MEhG} \int_{e \cap \Omega_0}
   \aver{\nabla v_h} \cdot \jump{u_h} \d{s}, \quad \sum_{e \in \MEhG}
   \int_{e \cap \Omega_0} \mu h_e^{-1} \jump{u_h} \cdot \jump{v_h}
   \d{s},
 \end{displaymath}
 respectively.
 \label{re_c0space}
\end{remark}
{
\begin{remark}
  The scheme \eqref{eq_dvariation} 
  can be termed as a symmetric
  interior penalty finite element method since $a_h(v_h, w_h) =
  a_h(w_h, v_h)$. In fact,  if we replace the trace terms $$\int_{e \cap \Omega_0}
  \aver{\nabla v_h} \cdot \jump{u_h} \d{s} \ \text{ and } \ \int_{\Gamma_K}
  \aver{\nabla v_h} \cdot \jump{u_h} \d{s}$$ 
  in \eqref{eq_bilinear} with 
  $$\int_{e \cap \Omega_0} -
  \aver{\nabla v_h} \cdot \jump{u_h} \d{s} \ \text{ and } \  \int_{\Gamma_K} -
  \aver{\nabla v_h} \cdot \jump{u_h} \d{s},$$
   respectively, then the  modified  scheme corresponds to a
  non-symmetric interior penalty method \cite{Riviere1999improved}.
  The   analysis for the error estimate \eqref{eq_DGerror} under the
  energy norm can be easily adapted for this case. Particularly, the
  non-symmetric method has a parameter-friendly feature, i.e. $\mu$
  can be selected as any positive number.
  \label{re_nipg}
\end{remark}
}

Next, we focus on the well-posedness of the discrete problem
\eqref{eq_dvariation}. For this goal, we introduce an energy norm
$\DGnorm{\cdot}$  defined by 
\begin{displaymath}
  \begin{aligned}
    \DGnorm{v_h}^2 := \sum_{K \in \MTho}  \|\nabla v_h &\|_{L^2(K \cap
    \Omega_0)}^2 + \sum_{e \in \MEho} h_e \| \aver{ \nabla v_h}
    \|_{L^2(e \cap \Omega_0)}^2 + \sum_{e \in \MEho} h_e^{-1} \|
    \jump{ v_h} \|_{L^2(e \cap \Omega_0)}^2 \\
    +& \sum_{K \in \MThG}  h_K \|\aver{ \nabla v_h}
    \|_{L^2(\Gamma_K)}^2 + \sum_{K \in \MThG} h_K^{-1} \| \jump{ v_h}
    \|_{L^2(\Gamma_K)}^2, \\
  \end{aligned}
\end{displaymath}
for any $v_h \in V_{h, 0} := V_{h, 0}^m + H^2(\Omega_0)$.

We give the following discrete trace estimate and  inverse estimate on
$\Gamma$, which are crucial in the forthcoming analysis.
\begin{lemma}
  There exists a   constant $C$ such that for any $K \in \MThG$ and
  $\alpha = 0, 1$,  it holds
  \begin{equation}
    \| D^\alpha v_h\|_{L^2( (\partial K)^0)} \leq C h_K^{-1/2} \|
    D^\alpha v_h \|_{L^2(K^\circ)}, \quad \forall v_h \in V_{h, 0}^m,  
      \label{eq_dtrace}
  \end{equation}
  \begin{equation}
    \| D^\alpha v_h \|_{L^2(K \cap \Omega_0)} \leq C \| D^\alpha v_h
    \|_{L^2(K^\circ)},  \quad \forall v_h \in V_{h, 0}^m,  
    \label{eq_dinverse}
  \end{equation}
  \label{le_dtrace}
   where $(\partial K)^0 = (\partial K \cap \Omega_0) \cup
  \Gamma_K$.
\end{lemma}
\begin{proof}
  Clearly, $(\partial K)^0 \in B(\bm{x}_{K^\circ}, C_\Delta
  h_{K^\circ})$ and the ball $B(\bm{x}_{K^\circ}, \rho_{K^\circ}) \in
  K^{\circ}$.  By \eqref{eq_Eh}, $(D^\alpha v_h)|_{ (\partial
  K)^0}$ has the same expression as $(D^\alpha v_h)|_{K^\circ}$, 
  and we deduce that 
  \begin{displaymath}
    \begin{aligned}
      \| D^\alpha v_h \|_{L^2( (\partial K)^{0})} &\leq |(\partial
      K)^0 |^{1/2} \| D^\alpha v_h \|_{L^\infty(B(\bm{x}_{K^\circ},
      C_\Delta h_{K^\circ}))} \\
      & \leq C |(\partial K)^0|^{1/2} |B(\bm{x}_{K^\circ},
      C_\Delta h_{K^\circ})|^{-1/2} \| D^\alpha v_h \|_{L^2(B(\bm{x}_{K^\circ},
      \rho_{K^\circ}))} \\
      & \leq C |(\partial K)^0|^{1/2} |B(\bm{x}_{K^\circ},
      C_\Delta h_{K^\circ})|^{-1/2}\| D^\alpha v_h \|_{L^2(K^\circ)} \\
      & \leq C h_K^{-1/2} \|D^\alpha v_h \|_{L^2(K^\circ)},
    \end{aligned}
  \end{displaymath}
 where in  the second inequality we have used the mesh regularity
 $C_\Delta h_K/ \rho_K \leq C$  and a scaling argument that  applies
 the inverse inequality $\| q_h \|_{L^\infty(B(0, 1))} \leq C \| q_h
 \|_{L^2(B(0, \rho_K/ (C_\Delta h_K)))}$ for any $q_h \in
 \mb{P}_m(B(0, 1))$ and the pullback with the bijective affine map
 from the ball $B(\bm{x}_{K^\circ}, C_\Delta h_{K^\circ})$ to $B(0,
 1)$, and in  the last inequality, we have used the mesh regularity
 $C_1 h_K \leq h_{K^\circ} \leq C_2 h_K$ and the estimate $| (\partial
 K)^0| \leq C h_K^{d - 1}$ \cite{Wu2012unfitted} due to the fact that
 $\Gamma$ is $C^2$-smooth or polygonal. Thus  the estimate
 \eqref{eq_dtrace} holds.  
  
  Similarly, we can obtain the estimate \eqref{eq_dinverse}. This 
  completes the proof.
\end{proof}

Based on above results, we are ready to prove that the bilinear form
$a_h(\cdot, \cdot)$ is bounded and coercive under the energy norm
$\DGnorm{\cdot}$.
\begin{lemma}
  Let $a_h(\cdot, \cdot)$ be defined as \eqref{eq_bilinear}, there
  exists a constant $C$ such that 
  \begin{equation}
    |a_h(u_h, v_h) |  \leq C \DGnorm{u_h} \DGnorm{v_h}, \quad \forall u_h, v_h
    \in 
    {V_{h, 0}}, 
    \label{eq_bounded}
  \end{equation}
  and with a sufficiently large $\mu$, there exists a constant $C$
  such that 
  \begin{equation}
    a_h(v_h, v_h) \geq C \DGnorm{v_h}^2, \quad \forall v_h \in V_{h,
    0}^m.
    \label{eq_coercive}
  \end{equation}
  \label{le_bc}
\end{lemma}
\begin{proof}
  By Cauchy-Schwarz inequality, we directly have 
  \begin{displaymath}
    \begin{aligned}
      - 2 \int_{e \cap \Omega_0} \aver{\nabla u_h} \cdot \jump{v_h}
      \d{s} \leq h_e \| \aver{\nabla u_h} \|_{L^2(e \cap \Omega_0)} +
      h_e^{-1} \| \jump{v_h} \|_{L^2(e \cap \Omega_0)}.
    \end{aligned}
  \end{displaymath}
  The other terms in \eqref{eq_bilinear} can  be bounded analogously. Thus,  the boundedness
  \eqref{eq_bounded} follows  from the
  definition of $\DGnorm{\cdot}$.

  The rest is to prove the coercivity \eqref{eq_coercive}. We
  introduce a weaker norm $\DGsnorm{\cdot}$, which is more
  {natural} for analysis, 
  \begin{displaymath}
    \DGsnorm{w_h}^2 := \sum_{K \in \MTho}  \|\nabla w_h \|_{L^2(K \cap
    \Omega_0)}^2 + \sum_{e \in \MEho} h_e^{-1} \| \jump{ w_h}
    \|_{L^2(e \cap \Omega_0)}^2 + \sum_{K \in \MThG} h_K^{-1} \|
    \jump{ w_h} \|_{L^2(\Gamma_K)}^2, 
  \end{displaymath}
  for $\forall w_h \in V_{h, 0}^m$. Then we state the equivalence
  between the norm $\DGnorm{\cdot}$ and the weaker norm
  $\DGsnorm{\cdot}$ restricted on the approximation space $V_{h,
  0}^m$. Obviously, it suffices to prove $\DGnorm{w_h} \leq C
  \DGsnorm{w_h}$.  To this end, we are required to bound the summation
  $\sum_{e \in \MEho} h_e \| \aver{ \nabla w_h} \|_{L^2(e \cap
  \Omega_0)}^2$ and $\sum_{K \in \MThG} h_K \| \aver{ \nabla w_h}
  \|_{L^2(\Gamma_K)}^2$ of $\DGnorm{w_h}$. By the trace estimate
  \eqref{eq_dtrace}, we obtain that 
  \begin{displaymath}
    \begin{aligned}
      \sum_{K \in \MThG} h_K \| \aver{\nabla w_h} \|_{L^2(\Gamma_K)}^2
      & \leq \sum_{K \in \MThG} C \| \nabla w_h \|_{L^2(K^\circ)}^2
      \leq C \sum_{K \in \MTho} \| \nabla w_h \|_{L^2(K \cap
      \Omega_0)}^2 \\
      & \leq C  \DGsnorm{w_h}^2. \\
    \end{aligned}
  \end{displaymath}
  Further, we consider the trace term defined on $e \in \MEho$. Let
  $e$ be shared by two neighbouring elements $K_1$ and $K_2$, we
  deduce that 
  \begin{displaymath}
    \begin{aligned}
      h_e \| \aver{\nabla w_h} \|_{L^2(e \cap \Omega_0)}^2 & \leq C
      h_e \left(\| \nabla w_h \|_{L^2(\partial K_1 \cap \Omega_0)}^2 +
      \| \nabla w_h \|_{L^2(\partial K_2 \cap \Omega_0)}^2 \right) \\
      & \leq C h_{K_1} \| \nabla w_h \|_{L^2(\partial K_1 \cap
      \Omega_0)}^2+ C h_{K_2} \| \nabla w_h \|_{L^2(\partial K_2
      \cap \Omega_0)}^2. \\
    \end{aligned}
  \end{displaymath}
  If $K_i \in \MThG  (i = 1 \ \text{or} \ 2)$ is a cut element, from the trace
  inequality \eqref{eq_dtrace}, there holds 
  \begin{displaymath}
    \begin{aligned}
      h_{K_i} \| \nabla w_h \|_{L^2(\partial K_i \cap
      \Omega_0)}^2  \leq C \| \nabla w_h \|_{L^2( (K_i)^{\circ})}^2.
    \end{aligned}
  \end{displaymath}
  If $K_i \in \MThoc$ is a non-interface element, the standard trace
  estimate directly gives that 
  \begin{displaymath}
    h_{K_i} \| \nabla w_h \|_{L^2(\partial K_i \cap \Omega_0)}^2 \leq
    C \| \nabla w_h \|_{L^2(K_i)}^2.
  \end{displaymath}
  Combining the above estimates shows that
  \begin{equation}
     \sum_{e \in \MEho} h_e \| \aver{ \nabla w_h} \|_{L^2(e \cap
     \Omega_0)}^2 \leq C \sum_{K \in \MTho} \| \nabla w_h \|_{L^2(K
     \cap \Omega_0)}^2 \leq C \DGsnorm{w_h}^2,
     \label{eq_averupper}
  \end{equation}
  which implies $\DGnorm{w_h} \leq C \DGsnorm{w_h}$, and also the
  equivalence between $\DGsnorm{\cdot}$ and  $\DGnorm{\cdot}$. This
  fact inspires us to prove the coercivity under the norm
  $\DGsnorm{\cdot}$.  From the Cauchy-Schwarz inequality and the
  estimate \eqref{eq_averupper}, for any $\varepsilon > 0$ there holds 
  \begin{displaymath}
    \begin{aligned}
      - 2\sum_{e \in \MEho} \int_{e \cap \Omega_0} \aver{\nabla v_h}
      \cdot \jump{v_h} \d{s} & \geq - \sum_{e \in \MEho} \varepsilon
      h_e \| \aver{\nabla v_h} \|_{L^2(e \cap \Omega_0)}^2 - \sum_{e
      \in \MEho}\varepsilon^{-1} h_e^{-1} \| \jump{v_h} \|_{L^2(e \cap
      \Omega_0)}^2 \\
      & \geq -C \varepsilon \sum_{K \in \MTho} \| \nabla v_h \|_{L^2(K
      \cap \Omega_0)}^2 - \sum_{e \in \MEho}\varepsilon^{-1} h_e^{-1}
      \| \jump{v_h} \|_{L^2(e \cap \Omega_0)}^2.  \\
    \end{aligned}
  \end{displaymath}
  The terms defined on $\Gamma_K$ can be bounded similarly, i.e.
  \begin{displaymath}
     - 2\sum_{K \in \MThG} \int_{\Gamma_K} \aver{\nabla v_h} \cdot
     \jump{v_h} \d{s} \geq -C \varepsilon \sum_{K \in \MTho} \| \nabla
     v_h \|_{L^2(K \cap \Omega_0)}^2 - \sum_{K \in
     \MThG}\varepsilon^{-1} h_K^{-1} \| \jump{v_h}
     \|_{L^2(\Gamma_K)}^2.
  \end{displaymath}
  By collecting all above results, we conclude that there exist
  constants $C_1$, $C_2$ and $C_3$ such that for any $\varepsilon > 0$
  there holds
  \begin{displaymath}
    \begin{aligned}
      a_h(v_h, v_h) \geq (1 - \varepsilon C_1)\sum_{K \in \MTho}
      \|\nabla v_h & \|_{L^2(K \cap \Omega_0)}^2 + (\mu -
      C_2/\varepsilon) \sum_{e \in \MEho} h_e^{-1} \| \jump{ v_h}
      \|_{L^2(e \cap \Omega_0)}^2 \\
      & + (\mu - C_3/\varepsilon) \sum_{K \in \MThG} h_K^{-1} \|
      \jump{ v_h} \|_{L^2(\Gamma_K)}^2.
    \end{aligned}
  \end{displaymath}
  Choose $\varepsilon = 1/(2C_1)$ and take a sufficiently large $\mu$,
  we arrive at $a_h(v_h, v_h) \geq C \DGsnorm{v_h}^2$, which gives the
  estimate \eqref{eq_coercive} and completes the proof.
\end{proof}
The Galerkin orthogonality also holds for the bilinear form
$a_h(\cdot, \cdot)$ and $l_h(\cdot)$.
\begin{lemma}
  Let $u \in H^2(\Omega)$ be the exact solution to  problem
  \eqref{eq_elliptic}, and let $u_h \in V_h^m$ be the numerical
  solution to   problem \eqref{eq_dvariation}, there holds
  \begin{equation}
    a_h(u - u_h, v_h) = l_h(v_h), \quad \forall v_h \in V_{h, 0}^m.
    \label{eq_Galerkinorth}
  \end{equation}
  \label{le_Galerkinorth}
\end{lemma}
\begin{proof}
  From the regularity of $u$, we have $\jump{u} |_e= 0$ for any face $e \in
  \MEh^0$. We bring $u$ into the bilinear form $a_h(\cdot, \cdot)$ and get
  \begin{displaymath}
    \begin{aligned}
      a(u, v_h) - l(v_h) & = \sum_{K \in \MTho}  \int_{K \cap
      \Omega_0} (\nabla u \cdot \nabla v_h - f v_h ) \d{x} \\
      & - \sum_{e \in \MEho} \int_{e \cap \Omega_0} \nabla u \cdot
      \jump{v_h} \d{s} - \sum_{K \in \MThG} \int_{\Gamma_K} \nabla u
      \cdot \jump{v_h} \d{s}.  \\
    \end{aligned}
  \end{displaymath}
  Applying integration by parts leads to
  \begin{displaymath}
    \begin{aligned}
      \sum_{K \in \MThoc} \int_{K} (\nabla u \cdot \nabla v_h - f v_h)
      \d{x} = & \sum_{e \in \MEhoc} \int_e \nabla u \cdot \jump{v_h}
      \d{s}, \\
      \sum_{K \in \MThG} \int_{K \cap \Omega_0} (\nabla u \cdot \nabla
      v_h - f v_h) \d{x} =  & \sum_{e \in \MEhG} \int_{e\cap
      \Omega_0} \nabla u \cdot \jump{v_h} \d{s} + \sum_{K \in \MThG}
      \int_{\Gamma_K} \nabla u \cdot \jump{v_h} \d{s}, \\
    \end{aligned}
  \end{displaymath}
  which indicate $a(u, v_h) - l(v_h) = 0$ and the Galerkin
  orthogonality \eqref{eq_Galerkinorth}. This completes the proof.
\end{proof}
The approximation error estimation under the error
measurement $\DGnorm{\cdot}$ requires the following trace inequality \cite{Hansbo2002discontinuous,Huang2017unfitted, Wu2012unfitted}:
\begin{lemma}
  There exists a constant $h_0$ independent of $h$ such that if $0 < h
  \leq h_0$,  there exists a constant
  $C$ such that 
  \begin{equation}
    \|w \|_{L^2(\Gamma_K)}^2 \leq C \left( h_K^{-1} \|w \|_{L^2(K)}^2
    + h_K \|w \|_{H^1(K)}^2 \right), \quad \forall w \in H^1(K), \quad
    \forall K \in \MThG.
    \label{eq_H1trace}
  \end{equation}
  \label{le_H1trace}
\end{lemma}

Moreover, we need to use the
Sobolev extension theory \cite{Adams2003sobolev} in the approximation analysis: there exists an
extension operator $E_0: H^s(\Omega_0) \rightarrow H^s(\Omega^*)(s
\geq 2)$ such that for any $w\in H^s(\Omega_0)$,
\begin{displaymath}
  (E_0 w)|_{\Omega_0} = w, \quad \|E_0 w \|_{H^q(\Omega^*)} \leq C \|
  w \|_{H^q(\Omega_0)}, \quad 2 \leq q \leq s.
\end{displaymath}
Combining Lemma \ref{le_Vhoaperror}, Lemma \ref{le_H1trace} and the
extension operator $E_0$, we give the following approximation estimate with
respect to $\DGnorm{\cdot}$:
\begin{theorem}
  For $0 < h \leq h_0$, there exists a constant $C$ such that 
  \begin{equation}
    \inf_{v_h \in V_{h, 0}^m} \DGnorm{u - v_h} \leq C h^m \|u
    \|_{H^{m+1}(\Omega_0)},  \quad \forall u \in H^{m+1}(\Omega_0).
    \label{eq_DGaperror}
  \end{equation}
  \label{th_DGaperror}
\end{theorem}
\begin{proof}
  Let $I_{h, 0} u$ be the Lagrange interpolant of $u$ into the space
  $V_{h, 0}^{m, \circ}$ and consider $v_h = E_{h, 0}(I_{h, 0} u)$.
  From Lemma \ref{le_Vhoaperror}, we have that 
  \begin{displaymath}
    \sum_{K \in \MTho} \|u - v_h \|_{H^q(K \cap \Omega_0)} \leq C
    h^{m+1-q} \|E_0 u \|_{H^{m+1}(\Omega^*)} \leq C
    h^{m+1-q} \| u \|_{H^{m+1}(\Omega_0)},
  \end{displaymath}
  with $q = 0, 1$. For any $e \in \MEho$, let $e$ be
  shared by $K_+$ and $K_-$, and by the standard trace estimate, we
  obtain 
  \begin{displaymath}
    \begin{aligned}
      \sum_{e \in \MEho} h_e \| \aver{ \nabla u  -  \nabla v_h}
      \|_{L^2(e \cap \Omega_0)}^2 &\leq \sum_{e \in \MEho} h_e \|
      \aver{ \nabla (E_0 u) - \nabla v_h} \|_{L^2(e)}^2 \\
      & \leq \sum_{e \in \MEho} C \left( \|  E_0 u - v_h
      \|_{H^1(K_+)}^2 + \|  E_0 u - v_h
      \|_{H^1(K_-)}^2 \right) \\
      & \leq C h^{2m} \| u \|_{H^{m+1}(\Omega_0)}^2. \\
    \end{aligned}
  \end{displaymath}
  Similarly, there holds
  \begin{displaymath}
    \sum_{e \in \MEho} h_e^{-1} \| \jump{u - v_h} \|_{L^2(e \cap
    \Omega_0)}^2 \leq C h^{2m} \| u \|_{H^{m+1}(\Omega_0)}^2. 
  \end{displaymath}
  By the trace estimate \eqref{eq_H1trace}, we have  
  \begin{displaymath}
    \begin{aligned}
      \sum_{K \in \MThG} h_K \| \aver{\nabla u - \nabla v_h}
      \|_{L^2(\Gamma_K)}^2 &\leq C \sum_{K \in \MThG} \| \nabla E_0 u -
      \nabla v_h \|_{H^1(K)}^2 \\
      & \leq C h^{2m} \| u \|_{H^{m+1}(\Omega_0)}^2
    \end{aligned}
  \end{displaymath}
  and  
  \begin{displaymath}
    \sum_{K \in \MThG} h_K^{-1} \| \jump{u - u_h} \|_{L^2(\Gamma_K)}^2
       \leq C h^{2m} \| u \|_{H^{m+1}(\Omega_0)}^2. \\
  \end{displaymath}
  Collecting all above estimates immediately leads to the error
  estimate \eqref{eq_DGaperror}, which completes the proof.
\end{proof}
Now, we are ready to give   {\it a priori} error estimates for our
method.
\begin{theorem}
  Let $u \in H^{m+1}(\Omega_0)$ be the exact
  solution to \eqref{eq_elliptic} and $u_h \in V_{h, 0}^m$ be the
  numerical solution to \eqref{eq_dvariation}, and let $a_h(\cdot, \cdot)$
  be defined as in \eqref{eq_bilinear} with a sufficiently large $\mu$. Then
  for $0 < h \leq h_0$, there exists a constant $C$ such that 
  \begin{equation}
    \DGnorm{u - u_h} \leq Ch^m \|u \|_{H^{m+1}(\Omega_0)},
    \label{eq_DGerror}
  \end{equation}
  and 
  \begin{equation}
    \| u - u_h  \|_{L^2(\Omega_0)} \leq C h^{m+1} \|u
    \|_{H^{m+1}(\Omega_0)}.
    \label{eq_L2error}
  \end{equation}
  \label{th_error}
\end{theorem}
\begin{proof}
  The proof follows from the standard Lax-Milgram framework. For any $v_h
  \in V_{h, 0}^m$, the boundedness \eqref{eq_bounded} and the
  coercivity \eqref{eq_coercive} give 
  \begin{displaymath}
    \begin{aligned}
      \DGnorm{u_h - v_h}^2 & \leq C a_h(u_h - v_h, u_h - v_h) = C
      a_h(u - v_h, u_h - v_h) \\
      & \leq C \DGnorm{u_h - v_h} \DGnorm{u - v_h}. \\
    \end{aligned}
  \end{displaymath}
  Applying the triangle inequality and the approximation estimate
  \eqref{eq_DGaperror} yields the  error estimate
  \eqref{eq_DGerror}.

  We prove the $L^2$ estimate by the dual argument. Let $\phi \in
  H^2(\Omega_0)$ solve the problem 
  \begin{displaymath}
    - \Delta \phi = u - u_h, \  \text{in } \Omega_0, \quad \phi = 0, \ 
    \text{on } \Gamma, 
  \end{displaymath}
  with the regularity estimate $\|\phi \|_{H^2(\Omega)} \leq C \| u -
  u_h \|_{L^2(\Omega)}$. Let $\phi_I$ be the linear interpolant
  of $\phi$ into the space $V_{h, 0}^{m, \circ}$, we have that
  \begin{displaymath}
    \begin{aligned}
      \|u - u_h\|_{L^2(\Omega)}^2  &= a_h(\phi, u - u_h) = a_h(\phi -
      E_{h, 0} \phi_I, u - u_h) \\
      & \leq C \DGnorm{\phi - E_{h, 0} \phi_I} \DGnorm{ u - u_h} \\ 
      & \leq C h \| \phi\|_{H^2(\Omega)} \DGnorm{ u - u_h} \\
      & \leq C h  \|u - u_h\|_{L^2(\Omega)} \DGnorm{ u - u_h}, \\
    \end{aligned}
  \end{displaymath}
  which implies \eqref{eq_L2error} and completes the proof.
\end{proof}
In the rest of this section, we give an upper bound of the condition number of final sparse
linear system, which is still independent of how the boundary $\Gamma$
cuts the mesh. The main ingredient is to prove a Poincar\'e type
inequality. 
\begin{lemma}
  For $0 < h \leq h_0$, there exist constants $C_1, C_2$ such that 
  \begin{equation}
    C_1 \| v_h \|_{L^2(\MThoc)} \leq \DGnorm{v_h} \leq C_2 h^{-1} \|
    v_h \|_{L^2(\MThoc)}, \quad \forall v_h \in V_{h, 0}^{m}.
    \label{eq_L2DGnorm}
  \end{equation}
  \label{le_L2DGnorm}
\end{lemma}
\begin{proof}
  By the inverse inequality \eqref{eq_dinverse}, we immediately have 
  \begin{equation}
    C \| v_h \|_{L^2(\Omega_0)} \leq  \| v_h \|_{L^2(\MThoc)} \leq \|
    v_h \|_{L^2( \Omega_0)}.
    \label{eq_L2L2L2}
  \end{equation}
  Let $\phi \in H^2(\Omega_0)$ be the solution of the   problem 
  \begin{displaymath}
    - \Delta \phi = v_h, \  \text{in } \Omega_0, \quad \phi = 0, \ 
    \text{on } \partial \Omega_0
  \end{displaymath}
  and satisfy $\| \phi \|_{H^2(\Omega_0)} \leq C \| v_h
  \|_{L^2(\Omega_0)}$.  Applying  integration by parts,  we find 
  that 
  \begin{displaymath}
    \begin{aligned}
      \| v_h &\|_{L^2(\Omega_0)}^2  = (-\Delta \phi,
      v_h)_{L^2(\Omega_0)} \\
      & = \sum_{K \in \MTho} (\nabla \phi, \nabla v_h)_{L^2(K \cap
      \Omega_0)} - \sum_{e \in \MEho} (\nabla \phi, \jump{v_h})_{L^2(e
      \cap \Omega_0)} - \sum_{K \in \MThG} (\nabla \phi,
      \jump{v_h})_{L^2(\Gamma_K)} \\
      & \leq C \DGnorm{v_h} \left( | \nabla  \phi|_{L^2(\Omega)}^2 +
      \sum_{e \in \MEho} h_e | \nabla \phi|_{L^2(e \cap \Omega_0)}^2 +
      \sum_{K \in \MThG} h_K | \nabla \phi|_{L^2(\Gamma_K)}^2
      \right)^{1/2}. \\
    \end{aligned}
  \end{displaymath}
  From Lemma \ref{le_H1trace} and the trace estimate, we deduce 
  \begin{displaymath}
    \begin{aligned}
      \sum_{K \in \MThG} h_K | \nabla \phi|_{L^2(\Gamma_K)}^2 \leq C
      \sum_{K \in \MThG} \| E_0 \phi \|_{H^2(K)}^2 \leq C\| \phi
      \|_{H^2(\Omega_0)}^2,
    \end{aligned}
  \end{displaymath}
  and
  \begin{displaymath}
    \begin{aligned}
      \sum_{e \in \MEho} h_e | \nabla \phi|_{L^2(e \cap \Omega_0)}^2
      \leq \sum_{e \in \MEho} h_e | \nabla (E_0 \phi)|_{L^2(e)}^2 \leq 
      C \sum_{K \in \MTho} \| E_0 \phi \|_{H^2(K)} ^2\leq C \| \phi
      \|_{H^2(\Omega_0)}^2.
    \end{aligned}
  \end{displaymath}
 These two inequalities, together with \eqref{eq_L2L2L2} and the
  regularity of $\phi$, imply $C \| v_h \|_{L^2(\MThoc)} \leq
  \DGnorm{v_h}$. Further, the inverse estimate \eqref{eq_dinverse}
  directly leads to 
  \begin{displaymath}
    \sum_{K \in \MTho} \| \nabla v_h \|_{L^2(K \cap \Omega_0)}^2 \leq 
    C h^{-2} \| v_h \|_{L^2(\MThoc)}^2. 
  \end{displaymath}
  Similar to the proof of the coercivity \eqref{eq_coercive}, by 
  the trace estimate and the inverse estimate we 
  can bound the trace terms of $\DGnorm{v_h}$ as follows:
  \begin{displaymath}
    \sum_{e \in \MEhoc} \left( h_e \| \aver{\nabla v_h} \|_{L^2(e \cap
    \Omega_0)}^2 + h_e^{-1}  \| \jump{v_h} \|_{L^2(e \cap \Omega_0)}^2
    \right) \leq C h^{-2} \| v_h \|_{L^2(\MThoc)}^2,
  \end{displaymath}
  \begin{displaymath}
    \sum_{K \in \MThG} \left( h_K \| \aver{\nabla v_h}\|_{L^2(\Gamma_K
    )}^2 + h_K^{-1}  \| \jump{v_h} \|_{L^2(\Gamma_K)}^2 \right)
    \leq C h^{-2} \| v_h \|_{L^2(\MThoc)}^2,
  \end{displaymath}
 which give  $\DGnorm{v_h} \leq C h^{-1} \| v_h \|_{L^2(\MThoc)}$ and
 finish the proof.
\end{proof}

Based on Lemma \ref{le_L2DGnorm}, the upper bound of the condition number of the discrete system can be obtained   similarly   as
in the standard finite element method
\cite{Brenner2007mathematical}. Let $\{ \phi_i\}(1 \leq i \leq N)$ be
the Lagrange basis of the space $V_{h, 0}^{m, \circ}$.  Clearly
$V_{h, 0}^m$ shares the same degrees of freedom and  
basis as that of $V_{h, 0}^{m, \circ}$. Let $A = (a_h(\phi_i, \phi_j))_{N \times N}$ be the resulting
stiff matrix and   $M = (\phi_i, \phi_j)_{N \times N}$ be the global
mass matrix. Then, for any vector $\bmr{v} \in \mb{R}^N$  there are
\begin{displaymath}
  a_h(v_h, v_h) = \bmr{v}^T A \bmr{v}, \quad (v_h, v_h) = \bmr{v}^T M
  \bmr{v}, \quad v_h = \sum_{i = 1}^N v_i \phi_i,
\end{displaymath}
where $\bmr{v} = (v_1, v_2, \ldots, v_N)^T$. 
\begin{theorem}
  For $0 < h \leq h_0$, there exists a constant $C$ such that 
  \begin{equation}
    \kappa(A) \leq Ch^{-2}.
    \label{eq_KA}
  \end{equation}
  \label{th_KA}
\end{theorem}
\begin{proof}
  We seek the lower and upper bounds of $(\bmr{v}^T A
  \bmr{v})/(\bmr{v}^T \bmr{v}) (\bmr{v} \neq \bm{0})$ to verify
  \eqref{eq_KA}. For any $\bmr{v}$, let $\bm{v}_h = \sum_{i = 1}^N v_i
  \phi_i$,  then $(\bmr{v}^T A \bmr{v})/(\bmr{v}^T \bmr{v})$ can be expressed as 
  \begin{displaymath}
    \begin{aligned}
      \frac{\bmr{v}^T A \bmr{v}}{\bmr{v}^T \bmr{v}} =
      \frac{a_h(v_h, v_h)}{\|v_h \|_{L^2(\Omega_0)}^2}
      \frac{\bmr{v}^T M \bmr{v}}{\bmr{v}^T \bmr{v}}.  
    \end{aligned}
  \end{displaymath}
  From Lemmas \ref{le_bc} and   \ref{le_L2DGnorm},  it follows 
  \begin{displaymath}
    C_1 \| v_h \|_{L^2(\Omega_0)}^2 \leq a_h(v_h, v_h) \leq C_2 h^{-2}
    \| v_h \|_{L^2(\Omega_0)}^2.
  \end{displaymath}
  Since $\bmr{v}$ corresponds to the degrees of freedom of the
  standard finite element space $V_{h, 0}^{m, \circ}$, we can know
  that
  \begin{displaymath}
    C_1 \|\bm{v}_h \|_{L^2(\MThoc)}^2 \leq \bmr{v}^T \bmr{v} \leq 
    C_2 \|\bm{v}_h \|_{L^2(\MThoc)}^2. 
  \end{displaymath}
  Putting all above results and the estimate \eqref{eq_L2L2L2}
  together, we arrive at
  \begin{displaymath}
    C_1 \leq \frac{\bmr{v}^T A \bmr{v}}{\bmr{v}^T \bmr{v}} \leq C_2
    h^{-2},
  \end{displaymath}
  which yields the bound \eqref{eq_KA} and completes the
  proof.
\end{proof}

We have shown that the unfitted scheme \eqref{eq_dvariation} for the   problem \eqref{eq_elliptic} 
 is stable and   can
achieve an arbitrarily high order accuracy  without any mesh adjustment or any special
stabilization technique.  In next section, we will extend this method
to the elliptic interface problem. 

\section{Approximation to Elliptic Interface Problem}
\label{sec_interface}
In this section, we are concerned with the following elliptic interface problem:
seek $u$ such that 
\begin{equation}
  \begin{aligned}
    - \nabla \cdot (\alpha \nabla u) & = f, && \text{in } \Omega_0
    \cup \Omega_1, \\
    u & = g, && \text{on } \partial \Omega, \\
    \jump{u} & = a \un, && \text{on } \Gamma, \\
    \jump{\alpha \nabla u} & = b, && \text{on  } \Gamma. \\
  \end{aligned}
  \label{eq_interfaceproblem}
\end{equation}
Here the definitions of domain $\Omega$ and $\Omega_0$ are consistent
with those in Section \ref{sec_preliminaries}. The domain $\Omega_1$ is
defined as $\Omega_1 := \Omega \backslash \overline{\Omega}_0$, and
$\alpha $ is 
{a} piecewise constant with  $\alpha|_{\Omega_i}=\alpha_i >0 \ (i=0,1)$.    $\Gamma$ is assumed to be $C^2$ smooth and
the domain $\Omega$ is regarded as being divided by the smooth interface
$\Gamma$ into two disjoint subdomains $\Omega_0$ and $\Omega_1$, where
$\Gamma = \partial \Omega_0$, see Fig.~\ref{fig_MThMTh0MTh1}. The data functions are assumed to
satisfy that $f \in L^2(\Omega)$, $g \in H^{3/2}(\partial \Omega)$, $a
\in H^{3/2}(\Gamma)$ and $b \in H^{1/2}(\Gamma)$, which make
\eqref{eq_interfaceproblem} possess a unique solution $u \in
H^2(\Omega_0 \cup \Omega_1)$. We refer to \cite{
Kellogg1972higher,Kellogg1976regularity} for more regularity results to such an interface
problem.

\begin{figure}[htp]
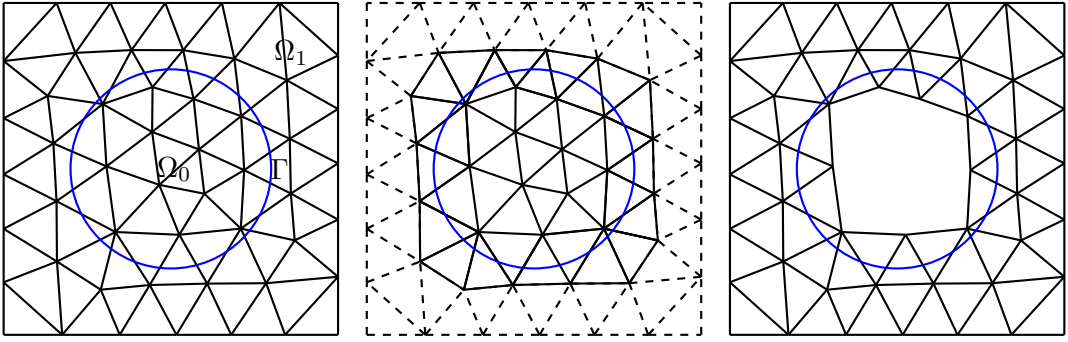

  \centering
  \begin{minipage}[t]{0.31\textwidth}
    \centering
    \begin{tikzpicture}[scale=2.2]
      \input{figure/MTh.tex}
      \draw[blue, thick] (0, 0) circle [radius=0.6]; 
      \node at (0.02, 0) {$\Omega_0$};
      \node at (0.72, 0.7) {$\Omega_1$};
      \node at (0.65, 0) {$\Gamma$};
    \end{tikzpicture}
  \end{minipage} 
  \begin{minipage}[t]{0.31\textwidth}
    \centering
    \begin{tikzpicture}[scale=2.2]
      \input{figure/eiMTh0.tex}
      \draw[blue, thick] (0, 0) circle [radius=0.6]; 
    \end{tikzpicture}
  \end{minipage}
  \begin{minipage}[t]{0.31\textwidth}
    \centering
    \begin{tikzpicture}[scale=2.2]
      \draw[thick,black] (-1, -1) -- (-0.65, -1);
\draw[thick,black] (-0.65, -1) -- (-1, -0.683333333333333);
\draw[thick,black] (-1, -1) -- (-1, -0.683333333333333);
\draw[thick,black] (-0.65, -1) -- (-0.681017500462872, -0.558204896788668);
\draw[thick,black] (-1, -0.683333333333333) -- (-0.681017500462872, -0.558204896788668);
\draw[thick,black] (-1, -0.361111111111111) -- (-0.681017500462872, -0.558204896788668);
\draw[thick,black] (-1, -0.361111111111111) -- (-1, -0.683333333333333);
\draw[thick,black] (-0.681017500462872, -0.558204896788668) -- (-0.41922890740978, -0.728075525788096);
\draw[thick,black] (-0.65, -1) -- (-0.41922890740978, -0.728075525788096);
\draw[thick,black] (-0.681017500462872, -0.558204896788668) -- (-0.68372439519797, -0.194463365458646);
\draw[thick,black] (-1, -0.361111111111111) -- (-0.68372439519797, -0.194463365458646);
\draw[thick,black] (-0.305555555555556, -1) -- (-0.41922890740978, -0.728075525788096);
\draw[thick,black] (-0.65, -1) -- (-0.305555555555556, -1);
\draw[thick,black] (-0.41922890740978, -0.728075525788096) -- (-0.331311786613425, -0.38047369369828);
\draw[thick,black] (-0.681017500462872, -0.558204896788668) -- (-0.331311786613425, -0.38047369369828);
\draw[thick,black] (-0.68372439519797, -0.194463365458646) -- (-0.331311786613425, -0.38047369369828);
\draw[thick,black] (-1, -0.0333333333333332) -- (-1, -0.361111111111111);
\draw[thick,black] (-1, -0.0333333333333332) -- (-0.68372439519797, -0.194463365458646);
\draw[thick,black] (-0.305555555555556, -1) -- (-0.126341050445663, -0.699940012213765);
\draw[thick,black] (-0.126341050445663, -0.699940012213765) -- (-0.41922890740978, -0.728075525788096);
\draw[thick,black] (-0.126341050445663, -0.699940012213765) -- (-0.331311786613425, -0.38047369369828);
\draw[thick,black] (-0.331311786613425, -0.38047369369828) -- (-0.385116189353426, 0.015281650560661);
\draw[thick,black] (-0.68372439519797, -0.194463365458646) -- (-0.385116189353426, 0.015281650560661);
\draw[thick,black] (-1, -0.0333333333333332) -- (-0.702474581762637, 0.153534370725956);
\draw[thick,black] (-0.702474581762637, 0.153534370725956) -- (-0.68372439519797, -0.194463365458646);
\draw[thick,black] (-0.702474581762637, 0.153534370725956) -- (-0.385116189353426, 0.015281650560661);
\draw[thick,black] (-0.305555555555556, -1) -- (0.0333333333333332, -1);
\draw[thick,black] (0.0333333333333332, -1) -- (-0.126341050445663, -0.699940012213765);
\draw[thick,black] (-0.126341050445663, -0.699940012213765) -- (0.0505799685141533, -0.396979325896565);
\draw[thick,black] (0.0505799685141533, -0.396979325896565) -- (-0.331311786613425, -0.38047369369828);
\draw[thick,black] (-1, 0.305555555555556) -- (-0.702474581762637, 0.153534370725956);
\draw[thick,black] (-1, 0.305555555555556) -- (-1, -0.0333333333333332);
\draw[thick,black] (0.0333333333333332, -1) -- (0.217303138112957, -0.691351145215774);
\draw[thick,black] (0.217303138112957, -0.691351145215774) -- (-0.126341050445663, -0.699940012213765);
\draw[thick,black] (0.217303138112957, -0.691351145215774) -- (0.0505799685141533, -0.396979325896565);
\draw[thick,black] (-0.702474581762637, 0.153534370725956) -- (-0.407739386523686, 0.390804950287482);
\draw[thick,black] (-0.407739386523686, 0.390804950287482) -- (-0.385116189353426, 0.015281650560661);
\draw[thick,black] (-0.702474581762637, 0.153534370725956) -- (-0.736374344891963, 0.439575866031186);
\draw[thick,black] (-1, 0.305555555555556) -- (-0.736374344891963, 0.439575866031186);
\draw[thick,black] (-0.736374344891963, 0.439575866031186) -- (-0.407739386523686, 0.390804950287482);
\draw[thick,black] (0.0333333333333332, -1) -- (0.361111111111111, -1);
\draw[thick,black] (0.361111111111111, -1) -- (0.217303138112957, -0.691351145215774);
\draw[thick,black] (0.0505799685141533, -0.396979325896565) -- (0.418488591484342, -0.359685730686808);
\draw[thick,black] (0.217303138112957, -0.691351145215774) -- (0.418488591484342, -0.359685730686808);
\draw[thick,black] (-1, 0.65) -- (-1, 0.305555555555556);
\draw[thick,black] (-1, 0.65) -- (-0.736374344891963, 0.439575866031186);
\draw[thick,black] (-0.407739386523686, 0.390804950287482) -- (-0.109476768044487, 0.492734073236539);
\draw[thick,black] (-0.736374344891963, 0.439575866031186) -- (-0.570731464576903, 0.699216048312196);
\draw[thick,black] (-0.570731464576903, 0.699216048312196) -- (-0.407739386523686, 0.390804950287482);
\draw[thick,black] (0.217303138112957, -0.691351145215774) -- (0.569823110531106, -0.688910763924027);
\draw[thick,black] (0.361111111111111, -1) -- (0.569823110531106, -0.688910763924027);
\draw[thick,black] (0.569823110531106, -0.688910763924027) -- (0.418488591484342, -0.359685730686808);
\draw[thick,black] (-1, 0.65) -- (-0.570731464576903, 0.699216048312196);
\draw[thick,black] (0.418488591484342, -0.359685730686808) -- (0.438537395964999, -0.00743031881140727);
\draw[thick,black] (0.13611041161168, 0.421257453635345) -- (-0.109476768044487, 0.492734073236539);
\draw[thick,black] (-0.235010081461715, 0.716088451794265) -- (-0.407739386523686, 0.390804950287482);
\draw[thick,black] (-0.235010081461715, 0.716088451794265) -- (-0.109476768044487, 0.492734073236539);
\draw[thick,black] (-0.235010081461715, 0.716088451794265) -- (-0.570731464576903, 0.699216048312196);
\draw[thick,black] (0.683333333333333, -1) -- (0.569823110531106, -0.688910763924027);
\draw[thick,black] (0.361111111111111, -1) -- (0.683333333333333, -1);
\draw[thick,black] (0.740807858779759, -0.430817581990893) -- (0.418488591484342, -0.359685730686808);
\draw[thick,black] (0.569823110531106, -0.688910763924027) -- (0.740807858779759, -0.430817581990893);
\draw[thick,black] (-0.683333333333333, 1) -- (-1, 0.65);
\draw[thick,black] (-0.683333333333333, 1) -- (-0.570731464576903, 0.699216048312196);
\draw[thick,black] (0.717699833536028, -0.148461769814811) -- (0.418488591484342, -0.359685730686808);
\draw[thick,black] (0.717699833536028, -0.148461769814811) -- (0.438537395964999, -0.00743031881140727);
\draw[thick,black] (0.420522653546179, 0.318138302188572) -- (0.13611041161168, 0.421257453635345);
\draw[thick,black] (0.420522653546179, 0.318138302188572) -- (0.438537395964999, -0.00743031881140727);
\draw[thick,black] (0.717699833536028, -0.148461769814811) -- (0.740807858779759, -0.430817581990893);
\draw[thick,black] (0.0736303644514813, 0.715481150106727) -- (0.13611041161168, 0.421257453635345);
\draw[thick,black] (0.0736303644514813, 0.715481150106727) -- (-0.109476768044487, 0.492734073236539);
\draw[thick,black] (-0.235010081461715, 0.716088451794265) -- (0.0736303644514813, 0.715481150106727);
\draw[thick,black] (-0.361111111111111, 1) -- (-0.235010081461715, 0.716088451794265);
\draw[thick,black] (-0.361111111111111, 1) -- (-0.570731464576903, 0.699216048312196);
\draw[thick,black] (-0.683333333333333, 1) -- (-1, 1);
\draw[thick,black] (-1, 1) -- (-1, 0.65);
\draw[thick,black] (-0.361111111111111, 1) -- (-0.683333333333333, 1);
\draw[thick,black] (0.683333333333333, -1) -- (1, -0.65);
\draw[thick,black] (1, -0.65) -- (0.569823110531106, -0.688910763924027);
\draw[thick,black] (1, -0.65) -- (0.740807858779759, -0.430817581990893);
\draw[thick,black] (0.379547875026244, 0.664627697913229) -- (0.420522653546179, 0.318138302188572);
\draw[thick,black] (0.379547875026244, 0.664627697913229) -- (0.13611041161168, 0.421257453635345);
\draw[thick,black] (0.0736303644514813, 0.715481150106727) -- (0.379547875026244, 0.664627697913229);
\draw[thick,black] (0.71125978886655, 0.181561221832313) -- (0.717699833536028, -0.148461769814811);
\draw[thick,black] (0.71125978886655, 0.181561221832313) -- (0.438537395964999, -0.00743031881140727);
\draw[thick,black] (0.71125978886655, 0.181561221832313) -- (0.420522653546179, 0.318138302188572);
\draw[thick,black] (-0.0333333333333332, 1) -- (-0.235010081461715, 0.716088451794265);
\draw[thick,black] (-0.0333333333333332, 1) -- (0.0736303644514813, 0.715481150106727);
\draw[thick,black] (-0.0333333333333332, 1) -- (-0.361111111111111, 1);
\draw[thick,black] (1, -0.305555555555556) -- (0.717699833536028, -0.148461769814811);
\draw[thick,black] (1, -0.305555555555556) -- (0.740807858779759, -0.430817581990893);
\draw[thick,black] (0.683333333333333, -1) -- (1, -1);
\draw[thick,black] (1, -1) -- (1, -0.65);
\draw[thick,black] (1, -0.65) -- (1, -0.305555555555556);
\draw[thick,black] (1, 0.0333333333333332) -- (0.71125978886655, 0.181561221832313);
\draw[thick,black] (1, 0.0333333333333332) -- (0.717699833536028, -0.148461769814811);
\draw[thick,black] (0.693308227213339, 0.534598649713033) -- (0.379547875026244, 0.664627697913229);
\draw[thick,black] (0.693308227213339, 0.534598649713033) -- (0.420522653546179, 0.318138302188572);
\draw[thick,black] (1, -0.305555555555556) -- (1, 0.0333333333333332);
\draw[thick,black] (0.71125978886655, 0.181561221832313) -- (0.693308227213339, 0.534598649713033);
\draw[thick,black] (0.305555555555556, 1) -- (0.0736303644514813, 0.715481150106727);
\draw[thick,black] (0.305555555555556, 1) -- (0.379547875026244, 0.664627697913229);
\draw[thick,black] (0.305555555555556, 1) -- (-0.0333333333333332, 1);
\draw[thick,black] (1, 0.0333333333333332) -- (1, 0.361111111111111);
\draw[thick,black] (1, 0.361111111111111) -- (0.71125978886655, 0.181561221832313);
\draw[thick,black] (1, 0.361111111111111) -- (0.693308227213339, 0.534598649713033);
\draw[thick,black] (0.65, 1) -- (0.693308227213339, 0.534598649713033);
\draw[thick,black] (0.65, 1) -- (0.379547875026244, 0.664627697913229);
\draw[thick,black] (0.65, 1) -- (0.305555555555556, 1);
\draw[thick,black] (1, 0.683333333333333) -- (0.693308227213339, 0.534598649713033);
\draw[thick,black] (1, 0.361111111111111) -- (1, 0.683333333333333);
\draw[thick,black] (1, 0.683333333333333) -- (0.65, 1);
\draw[thick,black] (1, 0.683333333333333) -- (1, 1);
\draw[thick,black] (1, 1) -- (0.65, 1);
      \draw[blue, thick] (0, 0) circle [radius=0.6]; 
    \end{tikzpicture}
  \end{minipage}
  \caption{The domain and the meshes $\MTh$ (left) /  $\MTh^0$
  (middle) /   $\MTh^1$ (right).}
  \label{fig_MThMTh0MTh1}
\end{figure}

Given the partition $\MTh$ (see the definition in Section
\ref{sec_preliminaries}), we introduce the following notations related
to the partition that will be used in this section (see
Fig.~\ref{fig_MThMTh0MTh1}):
\begin{displaymath}
  \begin{aligned}
    \MThl &:= \{ K \in \MTh \ | \ K \cap \Omega_1 \neq \varnothing\},
    \quad \MThlc :=  \{ K \in \MThl \ | \ K \subset \Omega_1 \}, \\
    \MEhl & := \{ e \in \MEh \ | \ e \cap \Omega_1 \neq \varnothing
    \}, \quad \MEhlc := \{ e \in \MEhl \ | \ e \subset \Omega_1 \}, \\
\end{aligned}
\end{displaymath}
and the notations $\MTho, \MThoc, \MThG, \MEhG$ follow the same
definitions as in Section \ref{sec_preliminaries}. Obviously, $\MThic
= \MThi \backslash \MThG (i = 0, 1)$.  For any element $K \in \MTh$
and any face $e \in \MEh$, we define 
\begin{displaymath}
  K^0 := K \cap \Omega_0, \quad K^1 := K \cap \Omega_1, \quad 
  e^0 := e \cap \Omega_0, \quad e^1 := e \cap \Omega_1. \\
\end{displaymath}
We suppose that   Assumption \ref{as_mesh2} holds individually for
$\MTho$ and $\MThl$, which reads
\begin{assumption}
  For any element $K \in \MThG$, there are two elements $K^\circ_0,
  K^\circ_1 \in \Delta(K)$ satisfying $K^\circ_0 \in \MThoc$ and
  $K^\circ_1 \in \MThlc$. 
  \label{as_mesh3}
\end{assumption}
The trace operators in \eqref{eq_traceG} on the interface $\Gamma$ are
specified as 
\begin{equation}
  \begin{aligned}
    \aver{v}|_{\Gamma_K} &:= \frac{1}{2}(v^0|_{\Gamma_K} +
    v^1|_{\Gamma_K}), \quad   \jump{v}|_{\Gamma_K} := (v^0 - v^1)\un,
    \\
    \aver{\bm{q}}|_{\Gamma_K} &:= \frac{1}{2}(\bm{q}^0|_{\Gamma_K} +
    \bm{q}^1|_{\Gamma_K}), \quad   \jump{\bm{q}}|_{\Gamma_K} :=
    (\bm{q}^0 - \bm{q}^1) \cdot \un, \quad 
  \end{aligned}
  \label{eq_ei_trace}
\end{equation}
for any $K \in \MThG$, where $v^0 = v|_{K^0}, v^1 = v|_{K^1}, \bm{q}^0
= \bm{q}|_{K^0}, \bm{q}^1 = \bm{q}|_{K^1}$ and $\un$ denotes the unit
normal vector on $\Gamma_K$ pointing to $\Omega_1$.  

Let us define the approximation space. For $i = 0, 1$, we let $V_{h,
i}^{m, \circ}$ be the $C^0$ finite element space or the discontinuous
finite element space with respect to the partition $\MThic$. Note that  the
spaces $V_{h, i}^{m, \circ}$ will be extended to $\MTh$ in a similar
way as in Section \ref{sec_preliminaries}. Let the extension operator
$E_h$ be piecewise defined as
\begin{equation}
  (E_h(v_{h, 0}, v_{h, 1}))|_K := \begin{cases}
    (v_{h, 0})|_K, & \forall K \in \MThoc, \\
    (v_{h, 1})|_K, & \forall K \in \MThlc, \\
    (E_{K^\circ_0} v_{h, 0})|_{K^0}, & \forall K \in \MThG, \\
    (E_{K^\circ_1} v_{h, 1})|_{K^1}, & \forall K \in \MThG, \\
  \end{cases}
  \label{eq_ei_Eh}
\end{equation}
for any $v_{h, 0} \in V_{h, 0}^{m, \circ}$ and any $v_{h, 1} \in V_{h,
1}^{m, \circ}$, where $K^\circ_i$ are the associated elements in
Assumption \ref{as_mesh3} and $E_K$ is the local extension operator given in
\eqref{eq_Ek}. We denote by $V_h^m$ the image space of $E_h$,
\begin{displaymath}
 V_h^m := \{ E_h(v_{h, 0}, v_{h, 1}) \ | \  \forall v_{h, 0} \in V_{h,
 0}^{m, \circ}, \ \forall v_{h, 1} \in V_{h, 1}^{m, \circ} \}.
\end{displaymath}
The space $V_h^m$ is actually the approximation space that will be
applied in numerically solving the interface problem
\eqref{eq_interfaceproblem}.  The space $V_h^m$ is a combination of
the extensions of spaces $V_{h, 0}^{m, \circ}$ and $V_{h, 1}^{m,
\circ}$. In addition, the degrees of freedom of $V_h^m$ are formed by
all degrees of freedom of $V_{h, 0}^{m, \circ}$ and $V_{h,
1}^{m, \circ}$, which are entirely located in $\Omega_0$ and
  $\Omega_1$, respectively. 

The discrete variational problem for 
\eqref{eq_interfaceproblem} reads: seek $u_h \in V_h^m$ such that 
\begin{equation}
  a_h(u_h, v_h) = l_h(v_h), \quad \forall v_h \in V_h^m, 
  \label{eq_ei_dvariance}
\end{equation}
where
\begin{eqnarray}
    a_h(u_h, v_h) &:=& \sum_{K \in \MTh} \int_{K^0 \cup K^1} \nabla u_h
    \cdot \nabla v_h \d{x}  \nonumber\\
    &&- \sum_{e \in \MEh} \int_{e^0 \cup e^1} \left( \aver{\alpha \nabla u_h}
    \cdot \jump{v_h}  +
    \aver{\alpha \nabla v_h} \cdot \jump{u_h} - \frac{\eta}{h_e} \jump{u_h} \cdot \jump{v_h}
    \right)\d{s}\\
    &&- \sum_{K \in \MThG} \int_{\Gamma_K}\left( \aver{\alpha \nabla u_h}
    \cdot \jump{v_h} +
    \aver{\alpha \nabla v_h} \cdot \jump{u_h} - \frac{\eta}{h_K} \jump{u_h} \cdot
    \jump{v_h} \right)\d{s}, \nonumber
  \label{eq_ei_bilinear}
\end{eqnarray}
for any $u_h, v_h \in V_h := V_h^m + H^2(\Omega_0 \cup \Omega_1)$, and
\begin{displaymath}
  \begin{aligned}
    l_h(v_h) := & \sum_{K \in \MTh} \int_{K^0 \cup K^1} f v_h \d{x} -
    \sum_{e \in \MEhb} \int_{e} \aver{\alpha \nabla v_h} \cdot \un g
    \d{s} + \sum_{e \in \MEhb} \int_{e} \frac{\eta}{h_e} g v_h \d{s}
    \\ 
    + & \sum_{K \in \MThG} \int_{\Gamma_K} b \aver{v_h} \d{s} -
    \sum_{K \in \MThG} \int_{\Gamma_K} \aver{\alpha \nabla v_h} \cdot
    \un a \d{s} + \sum_{K \in \MThG} \int_{\Gamma_K}
    \frac{\eta}{h_K} \jump{v_h} \cdot \un a \d{s}, \\
  \end{aligned}
\end{displaymath}
with the penalty parameter $\eta$. 

\begin{remark}
  If $V_{h, 0}^{m, \circ}$ and $V_{h, 1}^{m, \circ}$ are
  $C^0$ finite element spaces, then the terms 
  \begin{displaymath}
    \sum_{e \in \MEh} \int_{e^0 \cup e^1} \aver{\alpha \nabla u_h}
    \cdot \jump{v_h} \d{s}, \quad \sum_{e \in \MEh} \int_{e^0 \cup
    e^1} \aver{\alpha \nabla v_h} \cdot \jump{u_h} \d{s}, \quad
    \sum_{e \in \MEh} \int_{e^0 \cup e^1} \frac{\eta}{h_e} \jump{u_h}
    \cdot \jump{v_h} \d{s}
  \end{displaymath}
  in the bilinear form $a_h(\cdot, \cdot)$ can be simplified as 
  \begin{displaymath}
    \sum_{e \in \MEhG} \int_{e^0 \cup e^1} \aver{\alpha \nabla u_h}
    \cdot \jump{v_h} \d{s}, \quad \sum_{e \in \MEhG} \int_{e^0 \cup
    e^1} \aver{\alpha \nabla v_h} \cdot \jump{u_h} \d{s}, \quad
    \sum_{e \in \MEhG} \int_{e^0 \cup e^1} \frac{\eta}{h_e} \jump{u_h}
    \cdot \jump{v_h} \d{s}, 
  \end{displaymath}
  respectively.
  \label{re_ei_c0space}
\end{remark}
{
\begin{remark}
  As in Remark \ref{re_nipg}, the trace term $\int_{e^0 \cup e^1}
  \aver{\alpha \nabla v_h} \cdot \jump{u_h} \d{s}$ and
  $\int_{\Gamma_K} \aver{\alpha \nabla v_h} \cdot \jump{u_h} \d{s}$ in
  \eqref{eq_bilinear} can also be substituted respectively with $\int_{e^0 \cup e^1
  } - \aver{\alpha \nabla v_h} \cdot \jump{u_h} \d{s}$ and
  $\int_{\Gamma_K} - \aver{\alpha \nabla v_h} \cdot \jump{u_h} \d{s}$,
  which leads to  the non-symmetric scheme. The estimate
  \eqref{eq_ei_DGerror} can also be validated with any $\eta > 0$
  by following  the analysis of the symmetric case.
\end{remark}}

We introduce the energy norm $\DGenorm{\cdot}$ on $V_h$: 
\begin{displaymath}
  \begin{aligned}
    \DGenorm{v_h}^2 := \sum_{K \in \MTh} \| \nabla v_h \|_{L^2(K^0
    \cup K^1)}^2 + &\sum_{e \in \MEh} h_e \| \aver{\nabla v_h}
    \|_{L^2(e^0 \cup e^1)}^2 +  \sum_{e \in \MEh} h_e^{-1} \|
    \jump{v_h} \|_{L^2(e^0 \cup e^1)}^2 \\
    + & \sum_{K \in \MThG} h_K \| \aver{ \nabla v_h}
    \|_{L^2(\Gamma_K)}^2 + \sum_{K \in \MThG} h_K^{-1} \| \jump{v_h}
    \|_{L^2(\Gamma_K)}^2,
  \end{aligned}
\end{displaymath}
for any $v_h \in V_h$. Note that this norm is an direct extension of the norm
$\DGnorm{\cdot}$ defined in Section \ref{sec_elliptic}. 

The trace estimate \eqref{eq_dtrace} and the inverse estimate
\eqref{eq_dinverse} also hold for the space $V_h^m$:
\begin{lemma}
  For $i = 0, 1$,  there exists a
  constant $C$ such that for any element $K \in \MThG$,
  \begin{displaymath}
    \| D^\alpha v_h \|_{L^2( (\partial K)^i)} \leq C h_K^{-1/2} \|
    D^\alpha v_h \|_{L^2(K_i^\circ)}, \quad \forall v_h \in V_h^m,
    \quad \alpha = 0, 1, 
  \end{displaymath}
  \begin{displaymath}
    \|D^\alpha v_h \|_{L^2(K^i)} \leq C \| D^\alpha v_h
    \|_{L^2(K_i^{\circ})}, \quad \forall v_h \in V_h^m, \quad \alpha =
    0, 1,
  \end{displaymath}
  where $(\partial K)^i = (\partial K \cap \Omega_i) \cup \Gamma_K$.
  \label{le_ei_dtrace}
\end{lemma}
\begin{proof}
  For $i = 0$, this result is the same as Lemma \ref{le_dtrace}, and the  case $i = 1$ follows from 
  the same routine as in the proof of Lemma \ref{le_dtrace}. 
\end{proof}
From Lemma \ref{le_ei_dtrace}, the bilinear form $a_h(\cdot, \cdot)$
is bounded and coercive under the energy norm $\DGenorm{\cdot}$. 
\begin{lemma}
  Let $a_h(\cdot, \cdot)$ be defined as \eqref{eq_ei_bilinear}, there
  exists a constant $C$ such that 
  \begin{equation}
    |a_h(u, v)| \leq C \DGenorm{u} \DGenorm{v}, \quad \forall u, v \in
    V_h, 
    \label{eq_ei_bounded}
  \end{equation}
  and with a sufficiently large $\eta$, there exists a constant $C$
  such that 
  \begin{equation}
    a_h(v_h, v_h) \geq C \DGenorm{v_h}^2, \quad \forall v_h \in V_h^m.
    \label{eq_ei_coercive}
  \end{equation}
  \label{le_ei_bc}
\end{lemma}
\begin{proof}
  The proof is analogous to that of Lemma \ref{le_bc}.
  Applying the Cauchy-Schwarz inequality and  the definition
  of $\DGenorm{\cdot}$  immediately gives the   estimate
  \eqref{eq_ei_bounded}. 

  To obtain the coercivity \eqref{eq_ei_coercive}, we   introduce a
  weaker norm $\DGesnorm{\cdot}$ 
  \begin{displaymath}
    \DGesnorm{w_h}^2 := \sum_{K \in \MTh} \| \nabla w_h \|_{L^2(K^0
    \cup K^1)}^2 +  \sum_{e \in \MEh} h_e^{-1} \| \jump{w_h}
    \|_{L^2(e^0 \cup e^1)}^2 + \sum_{K \in \MThG} h_K^{-1} \|
    \jump{w_h} \|_{L^2(\Gamma_K)}^2,  
  \end{displaymath}
  for any $w_h \in V_h^m$.  The equivalence between $\DGnorm{\cdot}$
  and $\DGsnorm{\cdot}$ in Lemma \ref{le_bc} can be easily extended to
  $\DGenorm{\cdot}$ and $\DGesnorm{\cdot}$. Hence, it is sufficient to
  verify \eqref{eq_ei_coercive} under the norm $\DGesnorm{\cdot}$.
  From Lemma \ref{le_bc}, we actually have proven that for $i = 0$,
  there holds
  \begin{displaymath}
    \begin{aligned}
      \sum_{K \in \MTh^i} \| &\nabla v_h \|_{L^2(K^i)}^2 - \sum_{e \in
      \MEh^i} \int_{e^i} \aver{\alpha \nabla v_h} \cdot \jump{v_h}
      \d{s}  - \sum_{K \in \MThG} \int_{\Gamma_K \cap K^i}
      \aver{\alpha \nabla v_h} \cdot \jump{v_h} \d{s}\\
      + & \sum_{e \in \MEh^i} \eta h_e^{-1} \| \jump{v_h}
      \|_{L^2(e^i)}^2 + \sum_{K \in \MThG} \eta h_K^{-1} \| \jump{v_h}
      \|_{L^2(\Gamma_K)}^2 \\
      & \geq C \left( \sum_{K \in \MTh^i} \| \nabla v_h
      \|_{L^2(K^i)}^2 +  \sum_{e \in \MEh^i} \eta h_e^{-1} \|
      \jump{v_h} \|_{L^2(e^i)}^2 + \sum_{K \in \MThG} \eta h_K^{-1} \|
      \jump{v_h} \|_{L^2(\Gamma_K)}^2  \right),
    \end{aligned}
  \end{displaymath}
  with a sufficient large penalty $\eta$. Note that the above estimate
  can be shown to be valid for $i = 1$ by the same skill based
  on Lemma \ref{le_ei_dtrace}. Combining the above estimates for $i =
  0, 1$ and the definition of $\DGesnorm{\cdot}$ immediately yields
  the inequality \eqref{eq_ei_coercive}, which completes the proof.
\end{proof}
The   proof of Lemma \ref{le_Galerkinorth} also gives the Galerkin orthogonality   for this problem.
\begin{lemma}
  Let $u \in H^2(\Omega_0 \cup \Omega_1)$ be the exact solution to the
  problem \eqref{eq_interfaceproblem}, and let $u_h \in V_h^m$ be the
  numerical solution to the problem \eqref{eq_ei_dvariance}.  There
  holds
  \begin{displaymath}
    a_h(u - u_h, v_h) = l_h(v_h), \quad \forall v_h \in V_h^m.
  \end{displaymath}
  \label{le_ei_Galerkinorth}
\end{lemma}

For $i = 0, 1$, there exists an extension operator $E_i: H^s(\Omega_i)
\rightarrow H^s(\Omega)(s \geq 2)$ \cite{Adams2003sobolev} such that 
\begin{displaymath}
  (E_i w)|_{\Omega_i} = w, \quad \|E_i w \|_{H^q(\Omega)} \leq C \| w
  \|_{H^q(\Omega_i)}, \quad 2 \leq q \leq s.
\end{displaymath}
Then we state the approximation property of the space $V_h^m$. 
\begin{theorem}
  For $0 < h \leq h_0$, there exists a constant $C$ such that 
  \begin{equation}
    \inf_{v_h \in V_h^m} \DGenorm{u - v_h} \leq Ch^m \|u \|_{H^{m+1}(
    \Omega_0 \cup \Omega_1)}, \quad \forall u \in H^{m+1}(\Omega_0
    \cup \Omega_1).
    \label{eq_ei_aperror}
  \end{equation}
  \label{th_ei_aperror}
\end{theorem}
\begin{proof}
  The estimate \eqref{eq_ei_aperror} is based on the extension
  operators $E_i(i = 0, 1)$ and Lemma \ref{le_H1trace}, and the proof follows from the same line as in
  the proof of Theorem \ref{th_DGaperror}.
\end{proof}

Let us give    {\it a priori} error estimates for the proposed method. 
\begin{theorem}
  Let $u \in H^{m+1}(\Omega_0 \cup \Omega_1)$ be the exact solution to
  \eqref{eq_interfaceproblem} and $u_h \in V_h^m$ be the numerical
  solution to \eqref{eq_ei_dvariance}, and let $a_h(\cdot, \cdot)$ be
  defined as $\eqref{eq_ei_bilinear}$ with a sufficiently large
  $\eta$. Then  for $0 < h \leq h_0$, there exists a constant $C$ such that 
  \begin{equation}
    \DGenorm{u - u_h} \leq C h^m \| u \|_{H^{m+1}(\Omega_0 \cup
    \Omega_1)},
    \label{eq_ei_DGerror}
  \end{equation}
  and 
  \begin{equation}
    \| u - u_h \|_{L^2(\Omega)} \leq C h^{m+1} \|u
    \|_{H^{m+1}(\Omega_0 \cup \Omega_1)}.
    \label{eq_ei_L2error}
  \end{equation}
  \label{th_ei_error}
\end{theorem}
\begin{proof}
  The estimate  \eqref{eq_ei_DGerror} can be obtained by following the same line as in the proof of  
  \eqref{eq_DGerror} under the Lax-Milgram framework based on Lemma
  \ref{le_ei_bc}, \ref{le_ei_Galerkinorth} and Theorem
  \ref{th_ei_aperror}.  So  we only prove the $L^2$ error estimate
  \eqref{eq_ei_L2error}.  
  
  Let $\phi \in H^2(\Omega_0 \cup \Omega_1)$ be
  the solution to the problem 
  \begin{displaymath}
    \begin{aligned}
      - \nabla \cdot (\alpha \nabla \phi) & = u - u_h, && \text{in }
      \Omega_0 \cup \Omega_1, \\
      \phi & = 0, && \text{on } \partial \Omega, \\
      \jump{\phi} & = 0, && \text{on } \Gamma, \\
      \jump{\alpha \nabla \phi} & = 0, && \text{on  } \Gamma, \\
    \end{aligned}
  \end{displaymath}
  such that $\| \phi \|_{H^2(\Omega_0 \cup \Omega_1)} \leq C \|u - u_h
  \|_{L^2(\Omega)}$. We denote by $\phi_I$ the interpolant of $\phi$
  corresponding to the space $V_h^m$. Thus, it can be seen that
  \begin{displaymath}
    \begin{aligned}
       \| u - u_h \|_{L^2(\Omega)}^2 &= a_h(\phi, u - u_h) = a_h(\phi -
       \phi_I, u - u_h) \\
       & \leq C \DGenorm{\phi - \phi_I} \DGenorm{u - u_h} \\
       & \leq C h \|u - u_h\|_{L^2(\Omega)}  \DGenorm{u - u_h}, \\
    \end{aligned}
  \end{displaymath}
  which gives  \eqref{eq_ei_L2error}, and completes the proof.
\end{proof}
Ultimately, we present the estimate of the condition number for  the discrete system \eqref{eq_ei_dvariance}.
\begin{lemma}
  For $0 < h \leq h_0$, there exists constants $C_1, C_2$ such that 
  \begin{equation}
    C_1 \| v_h \|_{L^2(\MThoc \cup \MThlc)} \leq \DGenorm{v_h} \leq
    C_2 h^{-1} \| v_h \|_{L^2(\MThoc \cup \MThlc)}, \quad \forall v_h
    \in V_h^m.
    \label{eq_ei_L2DGnorm}
  \end{equation}
  \label{le_ei_L2DGnorm}
\end{lemma}
\begin{proof}
  From Lemma \ref{le_ei_dtrace} and the proof of Lemma \ref{le_L2DGnorm}, we conclude that 
  \begin{displaymath}
    C \| v_h \|_{L^2(\Omega)} \leq \| v_h \|_{L^2(\MThoc \cup
    \MThlc)} \leq \|v_h \|_{L^2(\Omega)}.
  \end{displaymath}
 Let $\phi \in H^2(\Omega_0 \cup
  \Omega_1)$ solve the interface problem 
  \begin{displaymath}
    \begin{aligned}
      - \nabla \cdot \nabla \phi & = v_h, && \text{in }
      \Omega_0 \cup \Omega_1, \\
      \phi & = 0, && \text{on } \partial \Omega, \\
      \jump{\phi} & = 0, && \text{on } \Gamma, \\
      \jump{ \nabla \phi} & = 0, && \text{on  } \Gamma, \\
    \end{aligned}
  \end{displaymath}
  with the regularity $\|\phi \|_{H^2(\Omega)} \leq C
  \|v_h \|_{L^2(\Omega)}$. From the integration by parts, we find that 
  \begin{displaymath}
    \begin{aligned}
      \| &v_h \|_{L^2(\Omega)}^2 = (- \nabla \cdot  \nabla
      \phi, v_h )_{L^2(\Omega)} \\
      &= \sum_{K \in \MTh} (\nabla \phi, \nabla v_h)_{L^2(K^0 \cup
      K^1)} - \sum_{e \in \MEh} (\nabla \phi, \jump{v_h})_{L^2(e^0
      \cup e^1)} - \sum_{K \in \MThG} (\nabla \phi,
      \jump{v_h})_{L^2(\Gamma_K)} \\
      & \leq C \DGenorm{v_h} \left( \sum_{K \in \MTh} \|\nabla \phi
      \|_{L^2(K^0 \cup K^1)}^2 + \sum_{e \in \MEh} h_e \| \nabla \phi
      \|_{L^2(e)}^2 + \sum_{K \in \MThG} h_K \| \nabla \phi
      \|_{L^2(\Gamma_K)}^2 \right)^{1/2}. 
    \end{aligned}
  \end{displaymath}
  From the trace estimate, we have 
  \begin{displaymath}
    \begin{aligned}
      \sum_{e \in \MEh} h_e \| \nabla \phi \|_{L^2(e)}^2 \
      \leq C \| \phi \|_{H^2(\Omega)}^2, \quad \sum_{K \in \MThG} h_K
      \| \nabla \phi \|_{L^2(\Gamma_K)}^2 \leq C \| \phi
      \|_{H^2(\Omega)}^2,
    \end{aligned}
  \end{displaymath}
  which give $\| v_h \|_{L^2(\Omega)} \leq C \DGenorm{v_h}$.
  Moreover, it is easy to verify $\DGenorm{v_h} \leq h^{-1} \|
  v_h \|_{L^2(\Omega)}$. This completes the proof.
\end{proof}
 
\begin{theorem}
  For $0 < h \leq h_0$, there exists a constant $C$ such that 
  \begin{equation}
    \kappa(A) \leq C h^{-2},
    \label{eq_ei_KA}
  \end{equation}
  \label{th_ei_KA}
 where $A$ denotes   the resulting stiff matrix  of the discrete system \eqref{eq_ei_dvariance}. 
\end{theorem}
\begin{proof}
  The estimate \eqref{eq_ei_KA} is a consequence of Lemma
  \ref{le_ei_L2DGnorm}; see the proof of Theorem \ref{th_KA}.
\end{proof}
The unfitted method in Section \ref{sec_elliptic} has been extended to the
interface problem. The used approximation space $V_h^m$ is easily implemented, since its
basis functions come from two common finite element spaces. This method neither requires  any constraint
on how the interface intersects the mesh nor includes any special
stabilization item.

\section{Numerical Results}
\label{sec_numericalresults}
In this section, a series of numerical results are presented to
illustrate the performance of the methods proposed in Sections
\ref{sec_elliptic} and   \ref{sec_interface}. In all tests, the
data functions $g$, $f$ in \eqref{eq_elliptic}, as well as the
functions $g$, $f$, $a$, $b$ in \eqref{eq_interfaceproblem}, are taken
suitably from the exact solution. The boundary or the interface for
each case is described by a level set function $\phi$. We note that
the scheme involves the numerical integration on the intersections of
the boundary/interface with elements. We refer to \cite{Cui2019quadratures,
Saye2015high} for some methods to seek the quadrature rules on the
curved domain, and  the codes are freely available online.

\subsection{Convergence Studies for Elliptic Problems}
We present several numerical examples to demonstrate the convergence
rates of the unfitted method \eqref{eq_dvariation} for the problem
\eqref{eq_elliptic}.  To obtain the approximation space $V_{h, 0}^m$,
the space $V_{h, 0}^{m, \circ}$ is selected to be the standard $C^0$
finite element space. The penalty parameter $\mu$ is taken as 
$\mu=3m^2+10$. 

\paragraph{\textbf{Example 1.}}
In this test,  we set the domain $\Omega_0:=\{(x,y)|\phi(x, y)< 0 \}$
to be a disk  (see Figure \ref{fig_ex1mesh}) with radius $r =
0.7$, that is, $\phi(x, y) = x^2 + y^2 - r^2$.  We take the background
mesh $\MTh$ that partitions the squared domain $\Omega = (-1, 1)^2$
into triangle elements with the mesh size $h = 1/5, \ldots, 1/40$, see
Figure \ref{fig_ex1mesh}.
The exact solution is given as 
\begin{displaymath}
  u(x, y) = \sin(2 \pi x) \sin(4 \pi y).
\end{displaymath}
We solve the discrete problem \eqref{eq_dvariation}   by $V_{h, 0}^{m}$
with   $1 \leq m \leq 3$.
The numerical errors under both the $L^2$ norm and the energy norm are
presented in Table ~\ref{tab_example1}. From the results, the optimal
convergence rates under $\|\cdot\|_{L^2(\Omega_0)}$ and
$\DGnorm{\cdot}$ are observed, which are in perfect agreement with the
theoretical estimates \eqref{eq_DGerror} and \eqref{eq_L2error} for
the 2D case.

\begin{figure}[htp]
  \centering
  \begin{minipage}[t]{0.46\textwidth}
    \centering
    \begin{tikzpicture}[scale=2.5]
      \centering
      \node at (0, 0) {$\Omega_0$};
      \draw[thick, red] (0, 0) circle [radius = 0.7];
      \draw[thick, black] (-1, -1) rectangle (1, 1);
    \end{tikzpicture}
  \end{minipage}
  \begin{minipage}[t]{0.46\textwidth}
    \centering
    \begin{tikzpicture}[scale=2.5]
      \centering
      \input{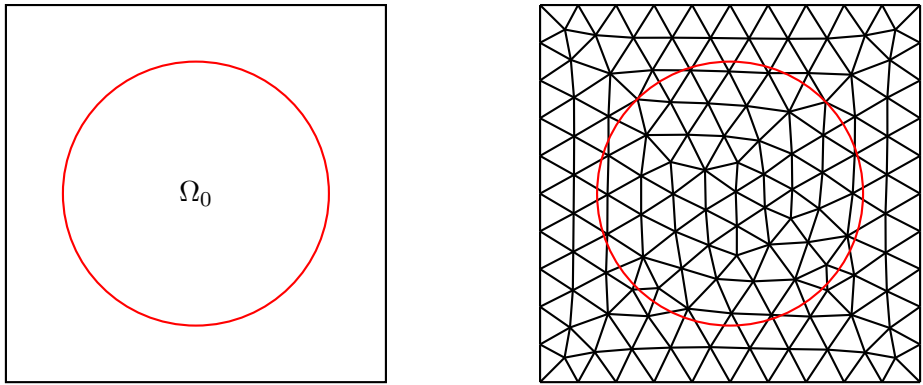}
      \draw[thick, red] (0, 0) circle [radius = 0.7];
    \end{tikzpicture}
  \end{minipage}
  \caption{The curved domain and the partition of Example 1.}
  \label{fig_ex1mesh}
\end{figure}

\begin{table}
  \centering
  \renewcommand\arraystretch{1.3}
  \begin{tabular}{p{0.3cm} | p{2.6cm} | p{1.6cm} | p{1.6cm} |
    p{1.6cm} | p{1.6cm} | p{1cm} }
    \hline\hline
    $m$ & $h$ & 1/5 & 1/10 & 1/20 & 1/40 & order \\
    \hline
    \multirow{2}{*}{$1$} & $\| \bm{u}-\bm{u}_h \|_{L^2(\Omega_0)}$
    & 4.647e-1 & 2.162e-1 & 4.402e-2 & 9.240e-3 & 2.25 \\
    \cline{2-7}
    & $\DGnorm{\bm{u} - \bm{u}_h}$ 
    & 6.868e-0 & 4.438e-0 & 1.992e-0 & 9.065e-2 & 1.13 \\ 
    \hline
    \multirow{2}{*}{$2$} & $\| \bm{u}-\bm{u}_h \|_{L^2(\Omega_0)}$
    & 1.966e-1 & 3.284e-2 & 2.415e-3 & 2.643e-4 & 3.19 \\
    \cline{2-7}
    & $\DGnorm{\bm{u} - \bm{u}_h}$  
    & 4.444e-0 &  1.249e-0 &  2.318e-1 &  4.912e-2 & 2.23 \\ 
    \hline 
    \multirow{2}{*}{$3$} & $\| \bm{u}-\bm{u}_h \|_{L^2(\Omega_0)}$
    & 7.924e-2 & 2.710e-3 &  2.117e-4 & 1.124e-5 & 4.23 \\
    \cline{2-7}
    & $\DGnorm{\bm{u} - \bm{u}_h}$  
    & 1.967e-0 &  1.795e-1 &  2.353e-2 &  2.629e-3 & 3.16 \\
    \hline\hline
  \end{tabular}
  \caption{The numerical errors of  Example 1.}
  \label{tab_example1}
\end{table}

\paragraph{\textbf{Example 2.}}
The second test is to solve the 2D elliptic problem defined on the
flower-like domain \cite{Massing2019stabilized} (see
Figure \ref{fig_ei_ex2mesh}), where $\Omega_0$ is governed by
the level set function $\phi < 0$, where
\begin{displaymath}
  \phi(r, \theta) = r - 0.6 - 0.2\cos(5 \theta), 
\end{displaymath}
with the polar coordinates $(r, \theta)$. The exact solution
\cite{Massing2019stabilized} reads
\begin{displaymath}
  u(x, y) = \cos(2 \pi x) \cos(2 \pi y) + \sin(2 \pi x) \sin(2 \pi y). 
\end{displaymath}
We solve \eqref{eq_dvariation}  on a series of triangular meshes ($h = 1/6, 1/12,
1/24, 1/48$) with   $m = 1, 2, 3$ on the domain $\Omega =
(-1, 1)^2$ (see Figure \ref{fig_ei_ex2mesh}).  The errors under two error
measurements are gathered in Table ~\ref{tab_example2}. For such a
curved domain, our method also demonstrates  that the errors $\|u - u_h
\|_{L^2(\Omega)}$ and $\DGnorm{u - u_h}$ approach zero at the  rates
$O(h^{m+1})$ and $O(h^m)$, respectively, which are well consistent
with the results   in Theorem \ref{th_error}. 

\begin{figure}[htp]
  \centering
  \begin{minipage}[t]{0.46\textwidth}
    \centering
    \begin{tikzpicture}[scale=2.5]
      \centering
      \node at (0, 0) {$\Omega_0$};
      \draw[thick, black] (-1, -1) rectangle (1, 1);
      \draw[thick, domain=0:360, red, samples=120] plot (\x:{(0.6 +
      cos(\x*5)*0.2)*1.02});
    \end{tikzpicture}
  \end{minipage}
  \begin{minipage}[t]{0.46\textwidth}
    \centering
    \begin{tikzpicture}[scale=2.5]
      \centering
      \input{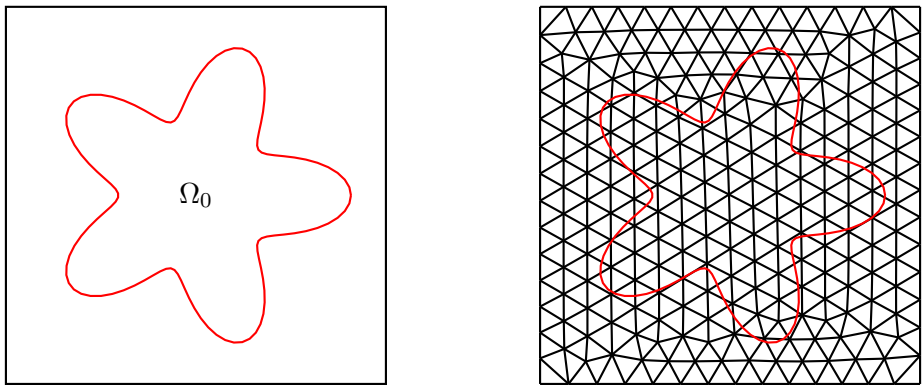}
      \draw[thick, domain=0:360, red, samples=120] plot (\x:{(0.6 +
      cos(\x*5)*0.2)*1.02});
    \end{tikzpicture}
  \end{minipage}
  \caption{The curved domain and the partition of Example 2.}
  \label{fig_ei_ex2mesh}
\end{figure}

\begin{table}
  \centering
  \renewcommand\arraystretch{1.3}
  \begin{tabular}{p{0.3cm} | p{2.6cm} | p{1.6cm} | p{1.6cm} |
    p{1.6cm} | p{1.6cm} | p{1cm} }
    \hline\hline
    $m$ & $h$ & 1/6 & 1/12 & 1/24 & 1/48 & order \\
    \hline
    \multirow{2}{*}{$1$} & $\| \bm{u}-\bm{u}_h \|_{L^2(\Omega_0)}$
    & 7.639e-1 & 2.342e-1 & 6.437e-2 & 1.320e-2 & 2.25 \\
    \cline{2-7} 
    & $\DGnorm{\bm{u} - \bm{u}_h}$ 
    & 7.740e-0 & 4.438e-0 &  1.985e-0 & 9.010e-1 & 1.13 \\ 
    \hline
    \multirow{2}{*}{$2$} & $\| \bm{u}-\bm{u}_h \|_{L^2(\Omega_0)}$
    & 1.702e-1 & 1.657e-2 & 1.785e-3 & 1.643e-4 & 3.44 \\
    \cline{2-7} 
    & $\DGnorm{\bm{u} - \bm{u}_h}$ 
    & 2.434e-0 & 5.163e-1&  9.593e-2 &  1.943e-2 & 2.30 \\
    \hline
    \multirow{2}{*}{$3$} & $\| \bm{u}-\bm{u}_h \|_{L^2(\Omega_0)}$
    & 6.725e-3 & 3.421e-4 & 2.158e-5 &  1.157e-6 & 4.22 \\ 
    \cline{2-7} 
    & $\DGnorm{\bm{u} - \bm{u}_h}$ 
    & 2.800e-1 & 2.649e-2 & 3.079e-3 &  3.128e-4 & 3.29 \\
    \hline\hline
  \end{tabular}
  \caption{The numerical errors of   Example 2.}
  \label{tab_example2}
\end{table}

\paragraph{\textbf{Example 3.}}
In this test, we solve a 3D elliptic problem defined in a spherical
domain $\Omega_0$  (see Figure \ref{fig_ex3interface}), whose corresponding level set function reads
\begin{displaymath}
  \phi(x, y, z) = (x - 0.5)^2 + (y-0.5)^2 + (z - 0.5)^2 - r^2, 
\end{displaymath}
with   radius $r = 0.35$. The exact solution $u$ is chosen as 
\begin{displaymath}
  u(x, y, z) = \cos(\pi x) \cos(\pi y) \cos(\pi z).
\end{displaymath}
We take a series of tetrahedral meshes, with  the mesh size $h = 
1/8, 1/16, 1/32, 1/64$, that cover the domain $\Omega = (0, 1)^3$. The numerical results    in Table~\ref{tab_example3} show
   that the proposed method still has the optimal
convergence rates for the errors $\|u - u_h \|_{L^2(\Omega)}$ and
$\DGnorm{u - u_h}$, respectively, which confirm our theoretical  estimates
\eqref{eq_DGerror} and \eqref{eq_L2error}.

\begin{figure}[htb]
  \centering
  \includegraphics[width=0.35\textwidth, height=0.35\textwidth]{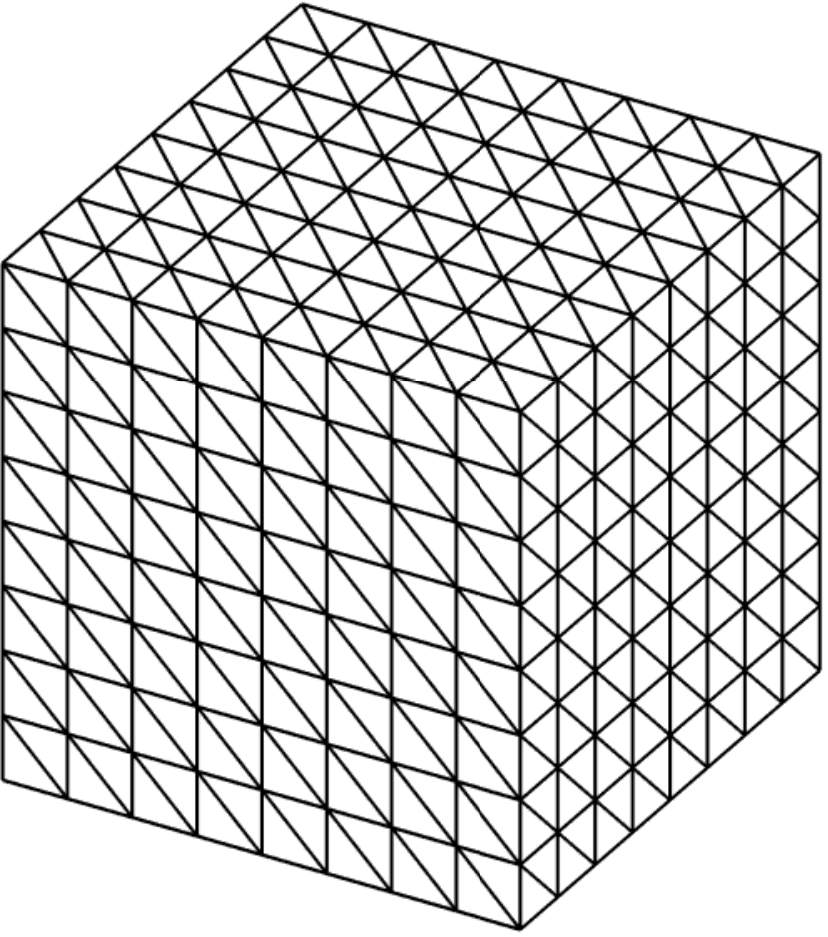}
  \hspace{15pt}
  \includegraphics[width=0.35\textwidth, height=0.35\textwidth]{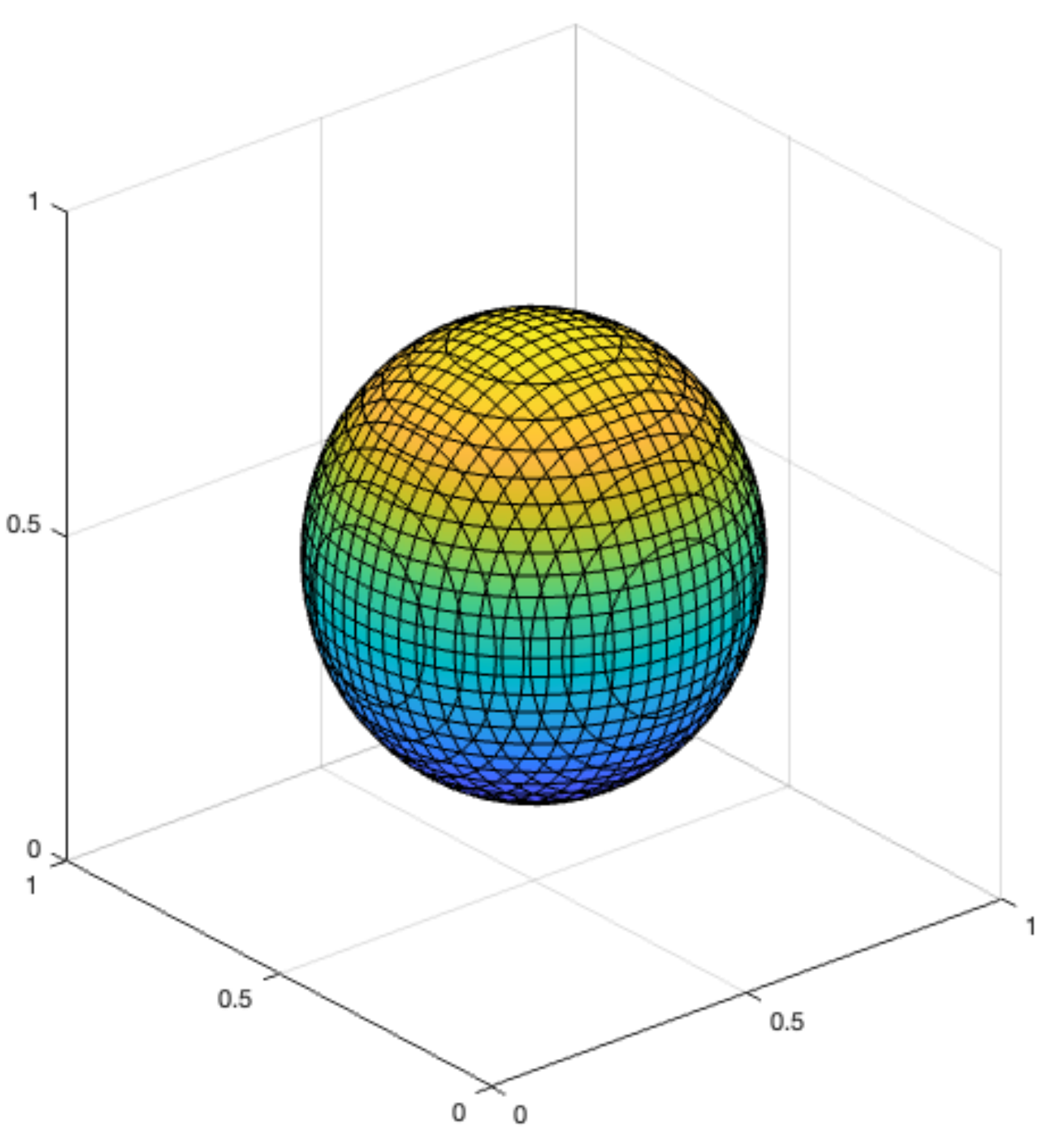}
  \caption{The spherical domain and the tetrahedral mesh of Example 3.}
  \label{fig_ex3interface}
\end{figure}

\begin{table}
  \centering
  \renewcommand\arraystretch{1.3}
  \begin{tabular}{p{0.3cm} | p{2.6cm} | p{1.6cm} | p{1.6cm} |
    p{1.6cm} | p{1.6cm} | p{1cm} }
    \hline\hline
    $m$ & $h$ & 1/8 & 1/16 & 1/32 & 1/64 & order \\
    \hline
    \multirow{2}{*}{$1$} & $\| \bm{u}-\bm{u}_h \|_{L^2(\Omega_0)}$
    & 8.357e-3 & 2.866e-3 & 1.042e-3 & 2.782e-4 & 1.91 \\
    \cline{2-7} 
    & $\DGnorm{\bm{u} - \bm{u}_h}$ 
    & 1.910e-1 & 1.143e-1 &  5.441e-2 & 2.481e-2 & 1.13 \\ 
    \hline
    \multirow{2}{*}{$2$} & $\| \bm{u}-\bm{u}_h \|_{L^2(\Omega_0)}$
    & 1.946e-3 & 8.168e-5 & 7.951e-6 & 7.897e-7 & 3.33 \\
    \cline{2-7} 
    & $\DGnorm{\bm{u} - \bm{u}_h}$ 
    & 5.882e-2 & 8.205e-3&  1.797e-3 &  4.063e-4 & 2.15 \\
    \hline
    \multirow{2}{*}{$3$} & $\| \bm{u}-\bm{u}_h \|_{L^2(\Omega_0)}$
    & 8.379e-5 & 2.828e-6 & 1.348e-7 &  7.699e-9 & 4.13 \\
    \cline{2-7} 
    & $\DGnorm{\bm{u} - \bm{u}_h}$ 
    & 3.793e-3 & 3.260e-4 & 3.357e-5 &  3.602e-6 & 3.22 \\
    \hline\hline
  \end{tabular}
  \caption{The numerical errors of  Example 3.}
  \label{tab_example3}
\end{table}

\subsection{Convergence Studies for Elliptic Interface Problems}
This subsection is devoted to verify the theoretical analysis of the
interface-unfitted scheme \eqref{eq_ei_dvariance}.
The spaces $V_{h, 0}^{m,
\circ}$ and $V_{h, 1}^{m, \circ}$ 
are taken as the $C^0$ finite element spaces. The penalty
parameter $\eta$ is selected as $3m^2+10$.

\paragraph{\textbf{Example 4.}}
This test is a 2D benchmark problem on $\Omega = (-1, 1)^2$
that contains a circular interface (see Figure \ref{fig_ei_eiex1mesh}), 
\begin{displaymath}
  \phi(x, y) = x^2 + y^2 - r^2=0, \quad \forall (x, y) \in (-1, 1)^2,
\end{displaymath}
with   radius $r = 0.5$. The piecewise coefficient $\alpha$ in
\eqref{eq_interfaceproblem} and the exact solution are respectively taken to be 
\begin{displaymath}
  \alpha = \begin{cases}
    b, & \phi(x, y) > 0, \\
    1, & \phi(x, y) < 0, \\
  \end{cases} \quad
  u(x, y) = \begin{cases}
    -\frac{1}{b}\left(  \frac{(x^2 + y^2)^2}{2} + x^2 + y^2
    \right) ,
    & \phi(x, y) > 0, \\
    \sin(2 \pi x) \sin(\pi y),  & \phi(x, y) < 0, \\
  \end{cases}
\end{displaymath}
with $b = 10$.
We adopt triangular meshes with $h
= 1/10, \ldots, 1/80$ with $1 \leq m \leq 3$. Numerical results  are
collected in Table~\ref{tab_example4}. We can observe that the proposed  unfitted method yields $ (m+1)$-th and
$m$-th convergence rates for 
the errors $\| u - u_h \|_{L^2(\Omega)}$ and $\DGenorm{u - u_h}$, respectively. This 
is in accordance with the predicted results in Theorem
\ref{th_ei_error}.

\begin{figure}[htp]
  \centering
  \begin{minipage}[t]{0.46\textwidth}
    \centering
    \begin{tikzpicture}[scale=2.5]
      \centering
      \node at (0, 0) {$\Omega_0$};
      \node at (0.7, 0.7) {$\Omega_1$};
      \draw[thick, black] (-1, -1) rectangle (1, 1);
      \draw[thick, red] (0, 0) circle [radius=0.5];
    \end{tikzpicture}
  \end{minipage}
  \begin{minipage}[t]{0.46\textwidth}
    \centering
    \begin{tikzpicture}[scale=2.5]
      \centering
      \input{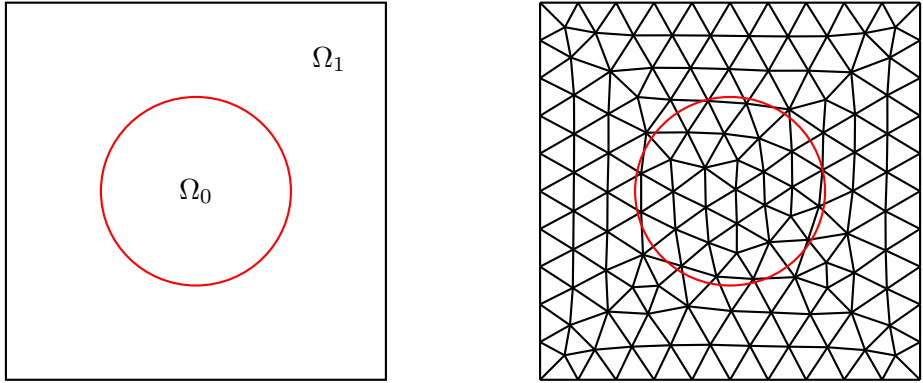}
      \draw[thick, red] (0, 0) circle [radius=0.5];
    \end{tikzpicture}
  \end{minipage}
  \caption{The interface and the partition of Example 4.}
  \label{fig_ei_eiex1mesh}
\end{figure}

\begin{table}
  \centering
  \renewcommand\arraystretch{1.3}
  \begin{tabular}{p{0.3cm} | p{2.6cm} | p{1.6cm} | p{1.6cm} |
    p{1.6cm} | p{1.6cm} | p{1cm} }
    \hline\hline
    $m$ & $h$ & 1/10 & 1/20 & 1/40 & 1/80 & order \\
    \hline
    \multirow{2}{*}{$1$} & $\| \bm{u}-\bm{u}_h \|_{L^2(\Omega)}$
    & 5.258e-2 &  1.558e-2 &  3.081e-3 & 6.888e-4 & 2.16 \\
    \cline{2-7} 
    & $\DGnorm{\bm{u} - \bm{u}_h}$ 
    & 9.881e-1 & 4.747e-1 & 1.953e-1 & 9.087e-2  & 1.10 \\
    \hline
    \multirow{2}{*}{$2$} & $\| \bm{u}-\bm{u}_h \|_{L^2(\Omega)}$
    & 1.055e-3 & 1.446e-4 & 1.130e-5 & 1.208e-6  & 3.23 \\
    \cline{2-7} 
    & $\DGnorm{\bm{u} - \bm{u}_h}$ 
    & 5.895e-2 & 1.450e-2 & 2.788e-3 & 6.590e-4  & 2.08 \\
    \hline
    \multirow{2}{*}{$3$} & $\| \bm{u}-\bm{u}_h \|_{L^2(\Omega)}$
    & 7.192e-5 & 3.885e-6 &  1.639e-7 &  9.199e-9 & 4.15 \\ 
    \cline{2-7} 
    & $\DGnorm{\bm{u} - \bm{u}_h}$ 
    & 5.176e-3 & 5.662e-4 &  5.108e-5 &  5.908e-6 & 3.11 \\ 
    \hline\hline
  \end{tabular}
  \caption{The numerical errors of   Example 4: $b=10$.}
  \label{tab_example4}
\end{table}
Further, we  also test the case,  by choosing $b = 1000$, that the coefficient has a large
jump. The numerical results are shown in
Table~\ref{tab_example4b}. By comparing the errors in
Table \ref{tab_example4} with Table~\ref{tab_example4b},  we demonstrate
the robustness of the proposed method for the problem involving a big
contrast on the interface.

\begin{table}
  \centering
  \renewcommand\arraystretch{1.3}
  \begin{tabular}{p{0.3cm} | p{2.6cm} | p{1.6cm} | p{1.6cm} |
    p{1.6cm} | p{1.6cm} | p{1cm} }
    \hline\hline
    $m$ & $h$ & 1/10 & 1/20 & 1/40 & 1/80 & order \\
    \hline
    \multirow{2}{*}{$1$} & $\| \bm{u}-\bm{u}_h \|_{L^2(\Omega)}$
    & 5.382e-2 &  1.819e-2 & 3.551e-3 & 7.490e-4 & 2.24 \\
    \cline{2-7} 
    & $\DGnorm{\bm{u} - \bm{u}_h}$ 
    & 9.818e-1 & 4.712e-1 &  2.002e-1 &  8.368e-2 & 1.23 \\
    \hline
    \multirow{2}{*}{$2$} & $\| \bm{u}-\bm{u}_h \|_{L^2(\Omega)}$
    & 1.019e-3 &  1.563e-4 & 1.132e-5 & 1.259e-6  & 3.17 \\
    \cline{2-7} 
    & $\DGnorm{\bm{u} - \bm{u}_h}$ 
    & 5.625e-2 & 1.547e-2 & 2.788e-3 & 6.698e-4 & 2.06 \\
    \hline
    \multirow{2}{*}{$3$} & $\| \bm{u}-\bm{u}_h \|_{L^2(\Omega)}$
    & 7.331e-5 & 4.302e-6 & 2.013e-7 &  1.048e-8 & 4.26 \\
    \cline{2-7} 
    & $\DGnorm{\bm{u} - \bm{u}_h}$ 
    & 5.258e-3 & 6.540e-4 & 6.369e-5 & 6.658e-6 & 3.25 \\
    \hline\hline
  \end{tabular}
  \caption{The numerical errors of   Example 4 with a large jump: $b=1000$.}
  \label{tab_example4b}
\end{table}

\paragraph{\textbf{Example 5.}} 
We consider an elliptic interface problem with a star interface
\cite{Zhou2006fictitious} (see Figure \ref{fig_ei_ex5mesh}), where
$\Gamma$ is parametrized with the polar coordinate $(r, \theta)$, 
\begin{displaymath}
  \phi(r, \theta) = r - \frac{1}{2} - \frac{\sin(5 \theta)}{7}.
\end{displaymath}
The domain is $\Omega = (-1, 1)^2$. The coefficient $\alpha$ and the
exact solution are selected to be 
\begin{displaymath}
\alpha = \begin{cases}
    10, & \phi(r, \theta) >0, \\
    1, & \phi(r, \theta) <  0, \\
  \end{cases} \quad
  u(r, \theta) = \begin{cases}
    0.1r^2 - 0.01\ln(2r), & \phi(r, \theta) >0, \\
    e^{r^2}, & \phi(r, \theta) < 0, \\
  \end{cases}   
  \end{displaymath}
  respectively. 
We display the numerical results in
Table~\ref{tab_example5}. Similar as the previous example, the optimal
convergence rates for the errors under the $L^2$ norm and the energy norm can
be still observed. 

\begin{figure}[htp]
  \centering
  \begin{minipage}[t]{0.46\textwidth}
    \centering
    \begin{tikzpicture}[scale=2.5]
      \centering
      \node at (0, 0) {$\Omega_0$};
      \node at (0.7, 0.7) {$\Omega_1$};
      \draw[thick, black] (-1, -1) rectangle (1, 1);
      \draw[thick, domain=0:360, red, samples=120] plot (\x:{(0.5 +
      sin(\x*5)/7)*1.0});
    \end{tikzpicture}
  \end{minipage}
  \begin{minipage}[t]{0.46\textwidth}
    \centering
    \begin{tikzpicture}[scale=2.5]
      \centering
      \input{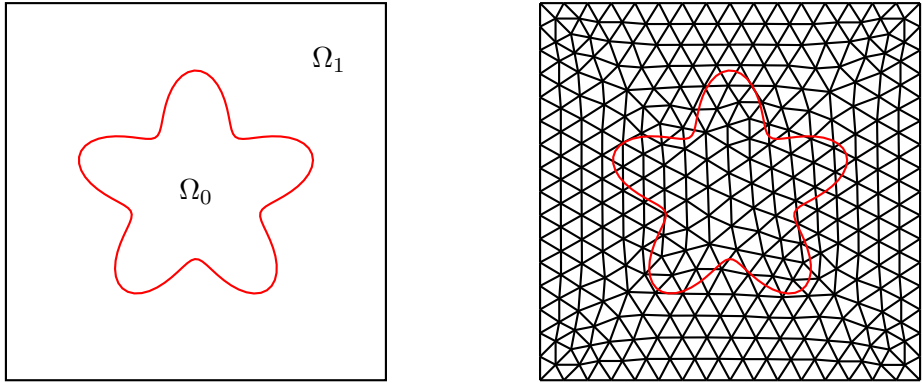}
      \draw[thick, domain=0:360, red, samples=120] plot (\x:{(0.5 +
      sin(\x*5)/7)*1.0});
    \end{tikzpicture}
  \end{minipage}
  \caption{The interface and the partition of Example 5.}
  \label{fig_ei_ex5mesh}
\end{figure}

\begin{table}
  \centering
  \renewcommand\arraystretch{1.3}
  \begin{tabular}{p{0.3cm} | p{2.6cm} | p{1.6cm} | p{1.6cm} |
    p{1.6cm} | p{1.6cm} | p{1cm} }
    \hline\hline
    $m$ & $h$ & 1/8 & 1/16 & 1/32 & 1/64 & order \\
    \hline
    \multirow{2}{*}{$1$} & $\| \bm{u}-\bm{u}_h \|_{L^2(\Omega)}$
    & 9.311e-3 & 2.269e-3 & 4.456e-4 & 9.326e-5 & 2.26 \\
    \cline{2-7} 
    & $\DGnorm{\bm{u} - \bm{u}_h}$ 
    & 1.795e-1 &  8.026e-2 & 3.488e-2 &  1.600e-2 & 1.12 \\ 
    \hline
    \multirow{2}{*}{$2$} & $\| \bm{u}-\bm{u}_h \|_{L^2(\Omega)}$
    & 5.070e-4 & 3.869e-5 & 3.360e-6 & 3.493e-7 & 3.26 \\
    \cline{2-7} 
    & $\DGnorm{\bm{u} - \bm{u}_h}$ 
    & 2.595e-2 & 2.905e-3 & 5.623e-4 & 1.229e-4  & 2.19 \\
    \hline
    \multirow{2}{*}{$3$} & $\| \bm{u}-\bm{u}_h \|_{L^2(\Omega)}$
    & 1.106e-5 &  6.759e-7 & 2.741e-8 & 1.408e-9 & 4.28 \\
    \cline{2-7} 
    & $\DGnorm{\bm{u} - \bm{u}_h}$ 
    & 1.361e-3 & 7.226e-5 & 6.342e-6 & 6.693e-7 & 3.23 \\
    \hline\hline
  \end{tabular}
  \caption{The numerical errors of   Example 5.}
  \label{tab_example5}
\end{table}

\paragraph{\textbf{Example 6.}} In the last example, we consider the
elliptic interface problem \eqref{eq_interfaceproblem} in three
dimensions with the coefficient $\alpha = 1$. The domain is the unit cube $\Omega = (0, 1)^3$ and the
interface is a smooth molecular surface of two atoms (see
Figure \ref{fig_ex6interface}), which is given by the level set
function \cite{Wei2018spatially, Li2018interface}, 
\begin{displaymath}
  \phi(x, y, z) = \left( (2.5(x - 0.5))^2 + (4(y-0.5))^2 + (2.5(z -
  0.5))^2 + 0.6 \right)^2 - 3.5(4(y-0.5))^2 - 0.6.
\end{displaymath}
The exact solution takes the form 
\begin{displaymath}
  u(x, y, z) = \begin{cases}
    e^{2(x + y + z)}, & \phi(x, y, z) > 0, \\
    \sin(2\pi x) \sin(2\pi y) \sin(2\pi z), & \phi(x, y, z) < 0. \\
  \end{cases}
\end{displaymath}
The initial mesh $\MTh$ is taken as
a tetrahedral with $h = 1/4$, and we solve the interface problem on a
series of successively refined meshes (see Figure
\ref{fig_ex3interface}). The convergence histories with
$1 \leq m \leq 3$ are reported in Table~\ref{tab_example6}, which show  that  both errors $\|u -
u_h \|_{L^2(\Omega)}$ and $\DGenorm{u - u_h}$ decrease to zero   at
their optimal convergence rates. This observation again
validates the theoretical predictions in Theorem \ref{th_ei_error}.

\begin{figure}[htb]
  \centering
  \includegraphics[width=0.4\textwidth]{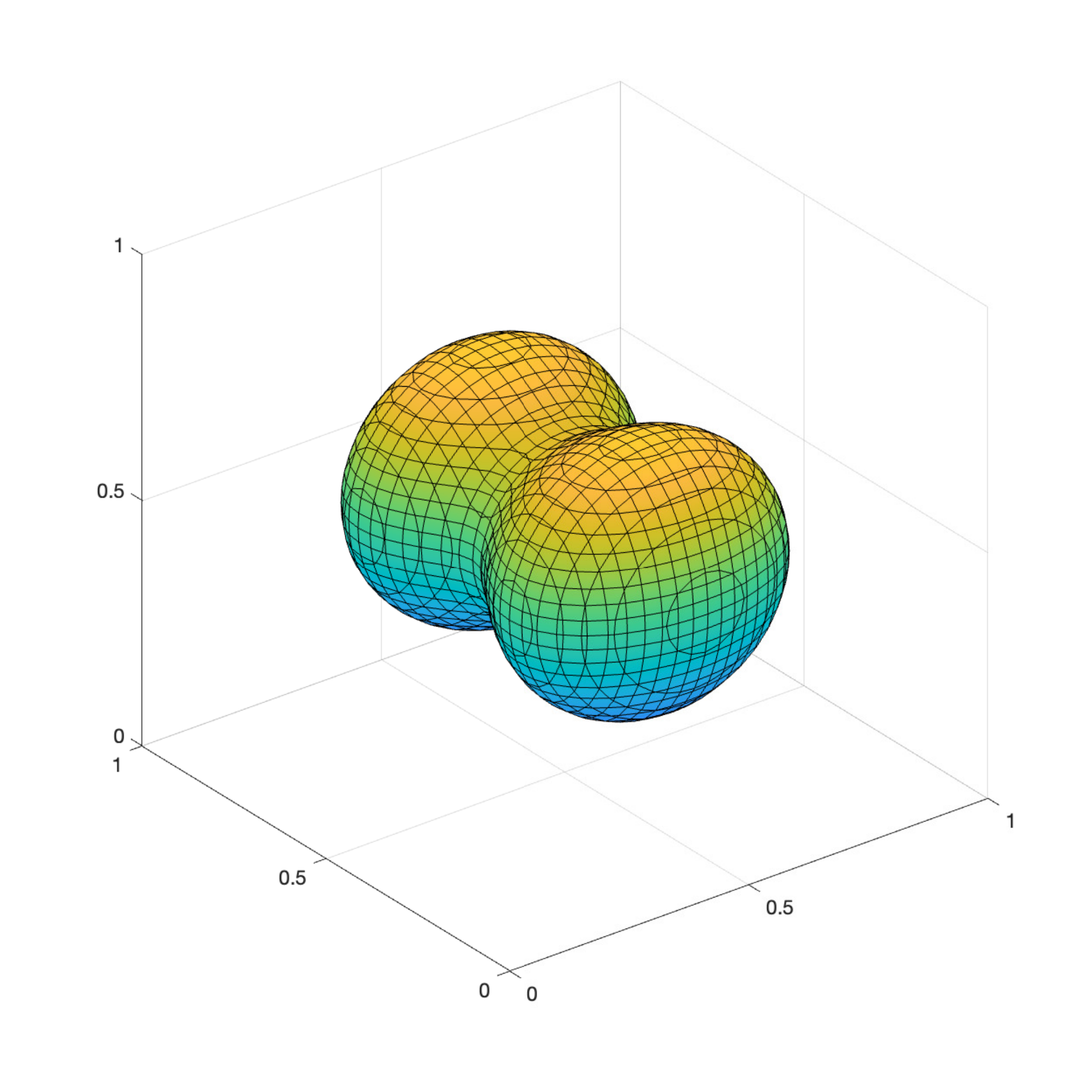}
  \hspace{10pt}
  \includegraphics[width=0.4\textwidth]{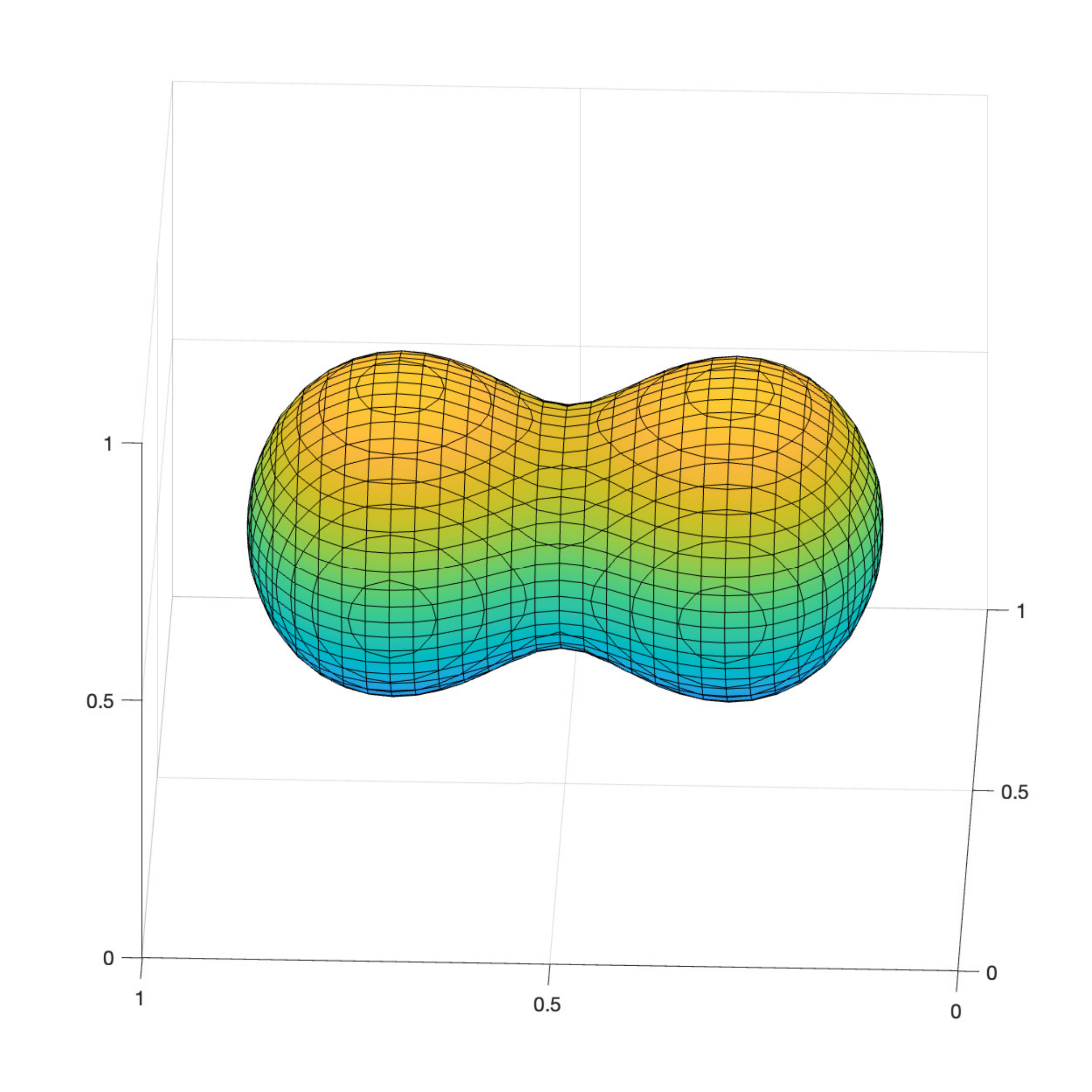}
  \caption{The interface of Example 6.}
  \label{fig_ex6interface}
\end{figure}

\begin{table}
  \centering
  \renewcommand\arraystretch{1.3}
  \begin{tabular}{p{0.3cm} | p{2.6cm} | p{1.6cm} | p{1.6cm} |
    p{1.6cm} | p{1.6cm} | p{1cm} }
    \hline\hline
    $m$ & $h$ & 1/4 & 1/8 & 1/16 & 1/32 & order \\
    \hline
    \multirow{2}{*}{$1$} & $\| \bm{u}-\bm{u}_h \|_{L^2(\Omega)}$
    & 1.705e-0 & 6.021e-1 & 1.260e-1 & 2.829e-2 & 2.15 \\
    \cline{2-7} 
    & $\DGnorm{\bm{u} - \bm{u}_h}$ 
    & 3.656e+1 &  1.783e+1 & 7.547e-0 &  3.675e-2 & 1.03 \\ 
    \hline
    \multirow{2}{*}{$2$} & $\| \bm{u}-\bm{u}_h \|_{L^2(\Omega)}$
    & 3.266e-1 & 2.326e-2 & 1.823e-3 & 1.930e-4 & 3.25 \\
    \cline{2-7} 
    & $\DGnorm{\bm{u} - \bm{u}_h}$ 
    & 4.948e-0 & 9.037e-1 & 1.698e-1 & 3.915e-2  & 2.12 \\
    \hline
    \multirow{2}{*}{$3$} & $\| \bm{u}-\bm{u}_h \|_{L^2(\Omega)}$
    & 3.609e-1 &  1.600e-3 & 7.349e-5 & 4.111e-6 & 4.16 \\
    \cline{2-7} 
    & $\DGnorm{\bm{u} - \bm{u}_h}$ 
    & 3.167e-0 & 7.372e-2 & 9.519e-3 & 1.182e-3 & 3.01 \\
    \hline\hline
  \end{tabular}
  \caption{The numerical errors of the Example 6.}
  \label{tab_example6}
\end{table}

\section{Conclusion}
\label{sec_conclusion}
We have developed unfitted finite element methods for   the elliptic boundary value 
  problem   and the elliptic  interface problem. The degrees of
freedom of the used approximation spaces are totally located in the elements that are not cut by the domain boundary and interface.
 The boundary condition  and the jump condition  
are weakly imposed by Nitsche's method. The stability near the boundary or the
interface does not require any stabilization
technique or any constraint on the mesh.  The optimal
convergence orders under the $L^2$ norm and the energy norm are
proved. In addition, we give  upper bounds of the condition numbers for
the two final linear systems. A series of numerical examples in two and
three dimensions are presented to validate our theoretical results.

\section*{Acknowledgements}
This work was supported   by National Natural Science Foundation of China (11971041, 12171340, 11771312).

\bibliographystyle{amsplain}
\bibliography{ref}

\end{document}